\documentclass[11pt]{article}

\usepackage{amssymb}
\usepackage{amsmath}
\usepackage{mathrsfs}
\usepackage{amsfonts}
\usepackage{mathtools}
\usepackage{vmargin}
\usepackage{amsthm}
\usepackage{graphicx}
\usepackage{authblk}
\usepackage{xfrac}
\usepackage{hyperref}
\usepackage{verbatim}

\usepackage[utf8]{inputenc}
\usepackage[T1]{fontenc}
\numberwithin{equation}{section}

\setmarginsrb{20mm}{20mm}{20mm}{20mm}{10mm}{10mm}{10mm}{10mm}

\newtheorem{defi}{Definition}[section]
\newtheorem{rem}[defi]{Remark}
\theoremstyle{plain}

\newtheorem{tw}[defi]{Theorem}

\newtheorem{lem}[defi]{Lemma}
\newtheorem{cor}[defi]{Corollary}
\newtheorem{prop}[defi]{Proposition}

\newcommand\bb[1]{\mathbf{#1}}

\renewcommand{\r}{\mathbb{R}}
\newcommand{\n}{\mathbb{N}}
\newcommand{\rn}{\mathbb{R}^d}
\newcommand{\tn}{\mathbb{T}^d}

\newcommand{\io}{\int_{\tn}}

\newcommand{\htwo}{H^2(\tn)}
\renewcommand{\hom}{H^1(\tn)}
\newcommand{\ld}{L^2(\tn)}
\newcommand{\loc}{\textnormal{loc}}
\newcommand{\ul}{\frac{1}{2}}
\newcommand{\sul}{\sfrac{1}{2}}
\newcommand{\ddt}{\frac{d}{dt}}

\newcommand{\hp}{H^{\perp}}
\newcommand{\ant}{\quad\textrm{ and }\quad}
\renewcommand{\div}{\operatorname{div}}
\DeclareMathOperator{\curl}{curl}
\newcommand{\F}{\mathcal{F}}
\newcommand{\T}{\mathcal{T}}
\let \epsilon\varepsilon
\let \phi\varphi
\newcommand{\initcond}{ \|\nabla \bb{v}_0\|^2_{\ld}+\|\Delta \bb{u}_0\|^2_{\ld}+\left\|\frac{\nabla\theta_0}{\sqrt{\theta_0}}\right\|^2_{\ld}\leq D}
\newcommand{\M}{\mathcal{M}}

\begin{document}
	\date{}
	\title{\bf Existence, uniqueness, and long-time asymptotic behavior of regular solutions in multidimensional thermoelasticity}
	
	\author[1]{Piotr Micha{\l} Bies\footnote{piotr.bies@pw.edu.pl, corresponding author}} \author[2,1]{Tomasz Cie\'slak\footnote{cieslak@impan.pl}} \author[3]{Mario Fuest\footnote{fuest@ifam.uni-hannover.de}} \author[3]{Johannes Lankeit\footnote{lankeit@ifam.uni-hannover.de}} \author[4]{Boris Muha\footnote{borism@math.hr}} \author[5]{Srdan Trifunovi\'c\footnote{srdjan.trifunovic@dmi.uns.ac.rs}}
	\affil[1]{Faculty of Mathematics and Information Sciences, Warsaw University of Technology, Ul. Koszykowa 75, 00-662 Warsaw, Poland, (+48) 22 621-93-12}
	\affil[2]{Institute of Mathematics, Polish Academy of Sciences, 00-656 Warsaw, Poland}
	\affil[3]{Leibniz Universit\"at Hannover, Institut f\"ur Angewandte Mathematik, Welfengarten 1, 30167 Hannover, Germany}
	\affil[4]{Department of Mathematics, Faculty of Science, University of Zagreb, Zagreb, Croatia}
	\affil[5]{Department for Mathematics and Informatics, University of Novi Sad, Novi Sad, Serbia}
	
	\maketitle
	\begin{abstract}
		\noindent
		We study a simplified nonlinear thermoelasticity model on two- and three-dimensional tori. A novel functional involving the Fisher information associated with temperature is introduced, extending the previous one-dimensional approach from the first two authors (SIAM J.\ Math.\ Anal.\ \textbf{55} (2023), 7024--7038)) to higher dimensions. Using this functional, we prove global/local existence of unique regular solutions for small/large initial data. Furthermore, we analyze the asymptotic behavior as time approaches infinity and show that the temperature stabilizes to a constant state, while the displacement naturally decomposes into two distinct components: a divergence-free part oscillating indefinitely according to a homogeneous wave equation and a curl-free part converging to zero. Analogous results for the Lam\'e operator are also stated.
	\end{abstract}
	\\${}$ 

	\section{Introduction.}
	In this paper, we investigate the existence, uniqueness, and long-term asymptotic behavior of solutions to a simplified nonlinear thermoelasticity model in higher-dimensional settings. Indeed, the objects of our interest are displacement $\bb{u}:[0,T)\times \tn\rightarrow \r^d$, $d=2,3$ and positive temperature $\theta:[0,T)\times \tn\rightarrow \r$ satisfying the following system of thermoelasticity representing the balance of momentum and heat:
	\begin{equation}\label{new:system}
		\begin{cases}
			\bb{u}_{tt}-\Delta \bb{u} =   -\mu\nabla\theta ,& \quad \text{ in } (0,T)\times \tn,\\
			\theta_t - \Delta \theta =- \mu\theta \div\bb{u}_t,& \quad \text{ in }  (0,T)\times\tn,\\
			\bb{u}(0,\cdot)=\bb{u}_0,~ \bb{u}_t(0,\cdot) = \bb{v}_0,~ \theta(0,\cdot)=\theta_0>0.&
		\end{cases}
	\end{equation}
	where $\mu>0$ is the heat expansion constant and $\mathbb{T}^d$ is the d-dimensional torus. Here, other constants have been set to $1$ for simplicity. 
	We refer the reader to \cite{slemrod, CMT} (or to \cite{Racke_book}) concerning the physical justification of the model. Let us however briefly mention that \eqref{new:system} is the simplest possible nonlinearly coupled model of the thermoelasticity phenomenon which is consistent with the laws of thermodynamics, where the Helmholtz free energy is assumed of the
	form 
	\begin{eqnarray*}%%%%%%%%%%%%%
		H(\theta, \nabla \bb{u})=\frac12|\nabla \bb{u}|^2+\theta -\mu\theta \div \bb{u}-\theta\log \theta.
	\end{eqnarray*}%%----------------------------%%
	The reader is referred to the 1D version of the model in \cite{BC, BC2} which was studied by the first two authors. 
	\begin{rem}
		We restrict our analysis to the toroidal domain to isolate and clearly illustrate the essential mathematical difficulties arising from the nonlinear coupling, while avoiding technical complications related to boundary conditions. The analysis of more general bounded domains with appropriate physical boundary conditions remains an interesting and important direction for future work. \\
	\end{rem}

	Notice that we can formally multiply $\eqref{new:system}_1$ with $\bb{u}_t$, sum up with $\eqref{new:system}_2$ and then integrate over $(0,t)\times \mathbb{T}^d$ to obtain the total energy conservation:
	\begin{align}\label{balance}
		\ul\io \bb{u}_t^2(t)dx+\ul\io |\nabla \bb{u}|^2(t)dx+\io\theta(t) dx=\ul\io \bb{v}_0^2dx+\ul\io |\nabla \bb{u}_0|^2dx+\io\theta_0 dx.
	\end{align}
	Due to expected positivity of the temperature, we can study the evolution of the entropy as well. 
	Correspondingly, the entropy evolution equation arising from dividing the heat equation \eqref{new:system}$_2$ by temperature $\theta$ reads as follows:
	\begin{align*}%\label{eq:tau}
		(\log\theta+\mu\div \bb{u})_t -\Delta\log\theta =|\nabla\log\theta|^2.
	\end{align*}
	Here, $\log\theta+\mu\div \bb{u}$ represents the entropy, $-\nabla\log\theta$ the entropy flux and $|\nabla\log\theta|^2$ the entropy production. This identity can be integrated over $\tn$ to obtain
	\begin{align}\label{ogrzprop}
		\frac{d}{dt}\io \log\theta(t) dx =\io |\nabla\log\theta|^2 (t) dx  
	\end{align}
	which can be integrated over $(0,t)$ and then subtracted from the total energy balance $\eqref{balance}$ to obtain total dissipation balance:
	\begin{eqnarray*}%%%%%%%%%%%%%
		&&\ul\io|\bb{u}_t|^2(t)dx+\ul\io |\nabla \bb{u}|^2(t)dx+\io(\theta-\log\theta)(t) dx + \int_0^t \io |\nabla\log\theta|^2 dx d\tau \nonumber\\
		&&=\ul\io \bb{v}_0^2dx+\ul\io |\nabla \bb{u}_0|^2dx+\io(\theta_0-\log\theta_0) dx. %\label{taul2}
	\end{eqnarray*}%%----------------------------%%
	\bigskip
	\begin{rem}%\label{uwaga}
		Note that $\theta-\log\theta>0$, so in the case of global solutions this (in)equality gives control of $\nabla\log\theta$ in $ L^2((0,\infty)\times\mathbb{T}^d)$. It is important to point out that this bound does not come from dissipation. In fact, there is no dissipation in the system and the only stabilizing effect comes from the entropy production. Surprisingly so, this will be enough to ensure the decay of temperature to a steady state and consequently the decay of the curl-free part of displacement to zero.
	\end{rem}

	\bigskip

	\medskip
	The following functional will play a key role in the this paper
	$$
	\F(\bb{u},\theta):=\ul\left(\io |\nabla \bb{u}_t|^2dx+\io(\Delta \bb{u})^2dx+\io\frac{|\nabla\theta|^2}{\theta}dx\right).
	$$
	\medskip
	
	\noindent
	It turns out that it has the following attractor set property: there is a constant $D>0$ depending on $\mu,d$ and $\mathbb{T}^d$ such that 
	\[
	\F(\bb{u}(t), \theta(t)))\leq D \implies \F(\bb{u}(t+\tau), \theta(t+\tau)))\leq D \text{ for all } \tau>0.
	\]
	This allows us to establish the existence of global-in-time solutions to \eqref{new:system} for initial data satisfying  $\F(\bb{u}_0, \theta_0))\leq D$. For general (possibly large) initial data, we will show the existence of local-in-time solutions. Moreover, the obtained regularity will be sufficient to ensure the uniqueness of solutions. Finally, we analyze the long-time asymptotic dynamics of global solutions. The displacement naturally splits into two distinct components via the Helmholtz decomposition: a divergence-free part, which oscillates indefinitely according to the homogeneous wave equation, and a curl-free part, which converges to zero as time approaches infinity. This asymptotic behavior is fundamentally connected to the structure of the coupling in the heat equation, which depends exclusively on $\div\bb{u}$. Consequently, the divergence-free component is unaffected by thermal dissipation and retains its oscillatory nature. Interestingly, this phenomenon differs from the one-dimensional setting studied in \cite{BC2}, as higher dimensions provide sufficient spatial freedom to sustain permanent oscillations within the divergence-free component, while the curl-free component remains stabilized by thermal effects.

	\section{Main results}
	In what follows, we shall study the solutions in the following sense.
	
	\begin{defi}\label{soldef}
		We say that $(\bb{u},\theta)$ is a solution to \eqref{new:system} on $(0,T)$, where $T \in (0, \infty]$, for initial data
		\begin{align*}
			\bb{u}_0\in \hom,\quad \bb{v}_0\in \ld,\quad \theta_0\in H^1(\tn), \quad \log \theta_0 \in L^1(\tn), \quad \nabla\sqrt\theta_0 \in \ld
		\end{align*}
		if:
		\begin{itemize}
			\item The functions $\theta$ and $\bb{u}$ satisfy
			\begin{align*}
				\bb{u}&\in L_{\loc}^{\infty}([0,T);\htwo)\cap W_{\loc}^{1,\infty}([0,T);\hom)\cap W_{\loc}^{2,\infty}([0,T);\ld),\nonumber\\
				\theta&\in L_{\loc}^2([0,T);\htwo)\cap H_{\loc}^1([0,T); \ld), \nonumber
			\end{align*}
			and
			\begin{align*}
				&\Delta \bb{u}\in L^{\infty}(0,T;\ld), \quad\nabla\bb{u}_t\in L^{\infty}(0,T;\ld), \\
				&\sqrt{\theta}\in L^{\infty}(0,T;H^1(\mathbb{T}^d)), \quad 
				\log\theta \in L^\infty(0,T;L^1(\mathbb{T}^d)), \quad \nabla\log\theta\in L^{2}(0,T;\ld).
			\end{align*}
			\item The initial data are attained, i.e.,
			\begin{align*}
				(\bb{u}(0), \bb{u}_t(0), \theta(0)) = (\bb{u}_0, \bb{v}_0, \theta_0) \quad \text{holds in $\hom \times \ld \times \ld$}.
			\end{align*}
			
			\item The momentum equation and the heat equation
			\begin{align}
				\bb{u}_{tt}- \Delta \bb{u}&=-\mu\nabla\theta \label{mom:eq}, \\
				\theta_t-\Delta\theta&=-\mu\theta \div \bb{u}_{t}, \label{heat:eq}
			\end{align}
			are satisfied a.e.\ on $(0,T)\times \tn$.
		\end{itemize}
	\end{defi}
	
	\bigskip 
	Our first main result establishes the existence and uniqueness of solutions, distinguishing explicitly between global existence, which is valid for sufficiently small initial data measured by our new functional, and local existence, which holds for general initial data.
	
	\begin{tw}[Existence and uniqueness]\label{main1}
		Let $d \in \{2, 3\}$.
		\begin{enumerate}
			\item \textbf{Global solution}:
			There is $D > 0$ depending only on $\mu,d$ and $\mathbb{T}^d$ such that if the initial data satisfy
			\begin{align}\label{initval}
				\bb{u}_0\in \htwo,\quad \bb{v}_0\in\hom,\quad 0\le\theta_0\in\hom, \quad \log\theta_0 \in L^1(\mathbb{T}^d), \quad \nabla\sqrt{\theta_0} \in L^2(\mathbb{T}^d)
			\end{align}
			and
			\begin{eqnarray}\label{initcond}
				\initcond,
			\end{eqnarray}
			then there exists a unique solution $(\bb u, \theta)$ of \eqref{new:system} on $(0, \infty)$ in the sense of Definition $\ref{soldef}$.
			
			\item \textbf{Local solution}:
			If the initial data satisfy \eqref{initval} but not necessarily \eqref{initcond},
			then there exists $T \in (0, \infty)$ depending on $\mu,d$ and $\mathbb{T}^d$ and a unique solution $(\bb u, \theta)$ of \eqref{new:system} on $(0, T)$ in the sense of Definition $\ref{soldef}$.
		\end{enumerate}
	\end{tw}
	
	\bigskip
	
	The second result demonstrates that uniform bounds on the temperature, both from above and below, are preserved during the evolution:
	
	\begin{tw}[Upper and lower bounds of temperature]\label{main2}
		Let $d \in \{2, 3\}$ and $(\bb{u},\theta)$ be the global-in-time solution obtained in the first part of Theorem~\ref{main1}. 
		Given $\underline \theta, \overline \theta > 0$, there exist $m, M > 0$ (depending in particular on $\underline \theta$ and $\overline \theta$, respectively) such that
		\begin{eqnarray*}%%%%%%%%%%%%%
			\theta_0\geq\underline{\theta} \quad \text{ a.e.\ on } \tn \quad &\implies \quad  \theta\geq m \quad \text{ a.e.\ on } (0,\infty)\times \mathbb{T}^d, \\
			\theta_0\leq\overline{\theta} \quad \text{ a.e.\ on } \tn \quad &\implies \quad  \theta\leq M \quad  \text{ a.e.\ on } (0,\infty)\times \mathbb{T}^d.
		\end{eqnarray*}%%----------------------------%%
	\end{tw}
	
	\bigskip
	
	An essential contribution of our analysis is the precise characterization of the long-term dynamics. In particular, we reveal how the interplay between mechanical displacement and thermal dissipation naturally leads to distinct behaviors: the divergence-free displacement indefinitely oscillates due to its wave nature, whereas the curl-free component, influenced by thermal diffusion, decays to zero as the temperature stabilizes uniformly. This decomposition of displacement into divergence-free and curl-free parts is accomplished via the Helmholtz decomposition, explicitly recalled in Section \ref{Sec:Helmholtz}.
	Let $\boldsymbol\chi$ be the curl-free part of $\bb u$, $\boldsymbol\chi_0$ be the curl-free of $\bb u_0$ and $\tilde{\boldsymbol\chi}_0$ be the curl-free part of $\bb v_0$.  
	\begin{tw}[Long-time dynamics]\label{main3}
		Let $d \in \{2, 3\}$ and assume that $(\bb u, \theta)$ is a global solution of \eqref{new:system} in the sense of Definition~\ref{soldef}
		starting from initial data $\bb u_0$, $\bb v_0$ and $\theta_0$, satisfying \eqref{initval}, \eqref{initcond} (for the $D$ given by Theorem~\ref{main1}) and $\theta_0^{-1} \in L^\infty(\tn)$.
		Then, the following
		convergences hold:
		\begin{alignat*}{2}
			\boldsymbol\chi(t,\cdot)&\to0 &&\quad\textrm{in }\hom \textrm{ as }t\to\infty,\\
			\boldsymbol\chi_t(t,\cdot)&\to 0 &&\quad\textrm{in }\ld \textrm{ as }t\to\infty,\\
			\theta(t,\cdot)&\to \theta_{\infty} &&\quad\textrm{in }\ld \textrm{ as }t\to\infty,
		\end{alignat*}
		where $\theta_{\infty}:=\left(\ul\io \tilde{\boldsymbol\chi}_0^2dx+\ul\io |\nabla\boldsymbol\chi_0|^2dx+\io\theta_0 dx\right)$.
	\end{tw}

	\begin{rem}
		It is worth noting that the divergence-free part of $\bb{u}$ does not interact with the curl-free part of $\bb u$ and $\theta$ so its evolution is completely determined by the wave equation. Therefore, it does not tend to a steady state and will continue to oscillate indefinitely (provided it is not zero initially).
	\end{rem}

	\subsection{Discussion of the literature}
	Thermoelasticity systems have long history, see for instance \cite{Racke_book}. The system we study, \eqref{new:system}, is a particular thermoelasticity system with the Helmholtz free energy given as 
	\begin{eqnarray*}%%%%%%%%%%%%%
		H(\theta, \nabla \bb{u})=\theta\log \theta-\theta +\mu\theta \div \bb{u}+\tfrac12|\nabla \bb{u}|^2.
	\end{eqnarray*}%%----------------------------%%
	The local-in-time unique solutions to the one-dimensional nonlinear thermoelastic system have been constructed by Slemrod in \cite{slemrod}. Next, the asymptotic behavior of solutions to the one-dimensional nonlinear thermoelastic system in the neighborhood of steady states has also been studied, see for instance \cite{slemrod, Racke_book, TC1}. Such results typically follow from energy estimates or fixed-point arguments for the coupled hyperbolic–parabolic system. Further refinements were provided by {\sc R. Racke} and collaborators, who established that in 1D, global solutions are exponentially stable when initial perturbations are small \cite{Racke_book}. Extending these results to higher dimensions proved challenging due to the presence of undamped rotational modes. In \cite{Racke_book}, global solutions in 2D and 3D under symmetry conditions eliminating rotational effects are obtained. A key challenge, however, is the question of \emph{global existence}: whether solutions remain well-defined for all time or develop singularities in finite time. It is well known that for sufficiently large initial data, smooth solutions to the highly nonlinear one-dimensional thermoelasticity system can experience finite-time blow-up due to nonlinear coupling effects, as shown by {\sc C. Dafermos} and {\sc L. Hsiao} \cite{DH}. {\sc W. Hrusa} and {\sc S. Messaoudi} extended this analysis to the physically more relevant case, proving that without smallness conditions, global smooth solutions generally do not exist \cite{HM}. Global weak solutions have been constructed in \cite{CMT} (see also a construction of solutions to the similar model in \cite{CGT}).  Finally, recently the global existence of unique solutions to the 1D case of \eqref{new:system} has been shown in \cite{BC}, introducing a novel functional related to the Fisher information to control the nonlinear coupling.     
	
	Another crucial aspect of thermoelastic systems is their \emph{asymptotic behavior}. Since thermal diffusion induces dissipation, one expects that perturbations of equilibrium states decay over time. In 1D, {\sc M. Slemrod} proved exponential decay for small solutions (see \cite{slemrod}). In \cite{BC2}, the full asymptotic behavior of the 1D minimal nonlinear thermoelasticity system, without restriction on initial data, is shown. However, the situation is more subtle in multiple dimensions due to the presence of undamped divergence-free modes. {\sc R. Racke} and others demonstrated that when the displacement’s \emph{curl} is initially zero, the solution exhibits full decay \cite{Racke_book}, but in general, divergence-free components persist as oscillations. This phenomenon is consistent with our findings: We rigorously show that in 2D and 3D, the displacement naturally splits into an oscillatory divergence-free part and a curl-free part that is dampened by thermal effects. See also \cite{Koch} and \cite{LebeauZuazua} for the convergence rates in the linear thermoelasticity.
	
	Beyond classical smooth solutions, researchers have also studied \emph{weak solutions} of thermoelastic systems. In three dimensions, \cite{CMT} constructed \emph{global weak solutions} via a half-Galerkin approximation and compactness arguments. This result ensures global solvability for large initial data but at the cost of lower regularity. A more generalized framework was considered by {\sc C. Christoforou, M. Galanopoulou}, and {\sc A. Tzavaras} \cite{CGT}, who studied \emph{measure-valued solutions} to capture fine-scale oscillations and singularities in thermoelastic systems. While such weak solutions are less regular, they provide global existence even in regimes where classical solutions may fail. A related study in \cite{TC1} analyzed a simplified one-dimensional thermoelasticity model, providing insights into the interplay between heat diffusion and elastic wave propagation in lower dimensions.
	
	Finally, for a comprehensive background on the mathematical theory of thermoelasticity, we refer the reader to the monograph by {\sc S. Jiang} and {\sc R. Racke} \cite{Racke_book}.
	
	\subsection{Plan of the paper}
	In Section \ref{druga}, in Lemma \ref{wonderfulformula} we present the new functional, which is a crucial point of our analysis. Next, after recalling some inequalities, in Lemma \ref{glestlem} we give an inequality, which plays an essential role in the proofs of our results. The section thereafter contains the construction of solutions, in particular a torus version of the half-Galerkin procedure borrowed from \cite{CMT}. Uniqueness of solutions is also shown. Section~\ref{sec:upper_lower_bdd} then shows time-independent $L^\infty$ upper and lower estimates of the temperature by a Moser-type iteration. In Section \ref{czwarta}, the asymptotics of solutions is studied. Next, the large-time behaviors of both displacement and temperature are given. In particular, in the latter the Helmholtz decomposition (see Theorem \ref{thdecomp}) and the second principle of thermodynamics \eqref{ogrzprop} are used. In Section~\ref{lamsec}, we state analogous results for the Lam\'e operator instead of the Laplacian in the wave equation. We also comment on the modifications of proofs required by the Lam\'e operator. 
	
	\section{A priori estimates}\label{druga}
	The first step towards constructing solutions consists of obtaining suitable a priori estimates. These will initially be derived under the assumption that solutions are sufficiently regular. Later, in sections dealing explicitly with the construction of solutions, we will explain how these a priori estimates can be rigorously justified within our procedures.

	\subsection{Energy and total dissipation}
	\begin{lem}\label{uni:est}
		Let $(\bb{u}, \theta)$ with $\theta > 0$ be a smooth solution of \eqref{new:system}. Then, we have the following total energy balance:
		\begin{align}\label{balance:smooth}
			\ul\io \bb{u}_t^2(t)dx+\ul\io |\nabla \bb{u}|^2(t)dx+\io\theta(t) dx=\ul\io \bb{v}_0^2dx+\ul\io |\nabla \bb{u}_0|^2dx+\io\theta_0 dx.
		\end{align}
		Moreover, the entropy equation holds:
		\begin{align}\label{eq:tau:smooth}
			(\log\theta+\mu\div \bb{u})_t -\Delta\log\theta =|\nabla\log\theta|^2,
		\end{align}
		and the total dissipation balance is satisfied
		\begin{eqnarray}%%%%%%%%%%%%%
			&&\ul\io|\bb{u}_t|^2(t)dx+\ul\io |\nabla \bb{u}|^2(t)dx+\io(\theta-\log\theta)(t) dx + \int_0^t \io |\nabla\log\theta|^2 dx d\tau \nonumber\\
			&&=\ul\io \bb{v}_0^2dx+\ul\io |\nabla \bb{u}_0|^2dx+\io(\theta_0-\log\theta_0) dx. \label{taul2:smooth}
		\end{eqnarray}%%----------------------------%%
	\end{lem}
	\begin{proof}
		The total energy balance $\eqref{balance:smooth}$ is obtained by multiplying the momentum equation $\eqref{mom:eq}$ with $\bb{u}_t$ and integrating over $(0,t)\times \mathbb{T}^d$ and summing up with the heat equation $\eqref{heat:eq}$ integrated over $(0,t)\times \mathbb{T}^d$. The entropy equation $\eqref{eq:tau:smooth}$ is obtained by dividing $\eqref{heat:eq}$ by $\theta$, which is integrated over $(0,t)\times \mathbb{T}^d$ and subtracted from $\eqref{balance:smooth}$ to obtain $\eqref{taul2:smooth}$.
	\end{proof}

	\subsection{Estimates of functional \texorpdfstring{$\mathcal{F}$}{F}}
	The following lemma states the main new estimate of the paper. It is based on Fisher information and is a higher-dimensional version of the functional in \cite{BC} which in our case takes the form
	\begin{align*}
		\F(\bb{u},\theta)=\ul\left(\io |\nabla \bb{u}_t|^2dx+\io(\Delta \bb{u})^2dx+\io\frac{|\nabla\theta|^2}{\theta}dx\right).
	\end{align*}
	It is a refined version of the formal estimate in \cite[Proposition 6.1]{BC}.

	\begin{lem}\label{wonderfulformula}
		Let $(\bb{u}, \theta)$ with $\theta > 0$ be a smooth solution of \eqref{new:system}. 
		Then, the equality
		\begin{align*}
			\frac{d}{dt}\F(\bb{u},\theta)=
			-\io\theta|\nabla^2\log\theta|^2dx-\frac\mu2\io\frac{|\nabla\theta|^2}\theta\div \bb{u}_tdx
		\end{align*}
		holds.
	\end{lem}
	\begin{proof}
		Dividing the heat equation $\eqref{new:system}_2$ by $2\theta^{\sul}$ yields
		\begin{align*}
			(\theta^{\sul})_t=\Delta\theta^{\sul}+\frac14\frac{|\nabla\theta|^2}{\theta^{\frac32}}-\frac{\mu}2\theta^{\sul}\div \bb{u}_t.
		\end{align*}
		We can now multiply this identity with $-4\Delta \theta^{\sul}$ and integrate over $\mathbb{T}^d$ to obtain
		\begin{align}\label{surow1}
			\ul&\ddt\io\frac{|\nabla\theta|^2}{\theta}dx\nonumber\\
			&=-4\io|\Delta \theta^{\sul}|^2dx-\io\Delta\theta^{\sul}\frac{|\nabla\theta|^2}{\theta^{\frac32}}dx+2\mu\io\Delta\theta^{\sul}(\theta^{\sul}\div \bb{u}_t)dx \eqqcolon I_1+I_2+I_3
		\end{align}
		We calculate
		\begin{align}
			I_2&=\io\nabla\theta^{\sul}\cdot\nabla\frac{|\nabla\theta|^2}{\theta^{\frac32}}dx=4\io\nabla\theta^{\sul}\cdot\nabla\frac{|\nabla\theta^{\sul}|^2}{\theta^{\sul}}dx \nonumber\\
			&=8\io\frac{\nabla\theta^{\sul} \cdot \nabla^2\theta^{\sul}\cdot(\nabla\theta^{\sul})^T}{\theta^{\sul}}dx-4\io\frac{|\nabla\theta^{\sul}|^4}{\theta}dx\label{surow3}
		\end{align}
		and
		\begin{align}\label{surow4}
			I_3=\mu\io\Delta\theta\div \bb{u}_tdx-\frac{\mu}2\io\frac{|\nabla\theta|^2}{\theta}\div \bb{u}_t dx.
		\end{align}
		Let us note that by \eqref{surow1} and \eqref{surow3},
		\begin{align}
			I_1+I_2=-4\io\sum_{k,l=1}^n\left((\theta^{\sul})_{x_kx_l}-\frac{(\theta^{\sul})_{x_k}(\theta^{\sul})_{x_l}}{\theta^{\sul}}\right)^2dx=-4\io\theta|\nabla^2\log\theta^{\ul}|^2dx. \label{surow5}
		\end{align}
		Plugging \eqref{surow4} and \eqref{surow5} into \eqref{surow1}, we get
		\begin{align}\label{wfe}
			\ul\ddt\io\frac{|\nabla\theta|^2}{\theta}dx=-4\io\theta|\nabla^2\log\theta^{\sul}|^2dx-\frac\mu2\io\frac{|\nabla\theta|^2}\theta\div \bb{u}_tdx+\mu\io\Delta\theta\div \bb{u}_tdx.
		\end{align}
		Next, multiplying the wave equation in \eqref{new:system} by $-\Delta \bb{u}_t$, integrating over $\tn$ and integrating by parts gives us
		\begin{align}\label{wfu1}
			\ul\ddt\left(\io|\nabla \bb{u}_t|^2dx+\io(\Delta \bb{u})^2dx\right)=\mu\io\nabla\theta\cdot\Delta \bb{u}_t dx.
		\end{align}
		Let us recall that for a vector field $w\colon\tn\to\rn$ (with $d \in \{2, 3\}$) we have
		\begin{align*}
			\Delta \bb w=-\curl\curl \bb w+\nabla\div \bb w.
		\end{align*}
		Hence, applying it to \eqref{wfu1} we obtain
		\begin{align*}
			\ul\ddt\left(\io|\nabla \bb{u}_t|^2dx+\io(\Delta \bb{u})^2dx\right)&=-\underbrace{\mu\io\nabla\theta\cdot\curl\curl \bb{u}_tdx}_{=0}+\mu\io\nabla\theta\cdot\nabla\div \bb{u}_t dx\\
			&= - \mu\io\Delta\theta \div \bb{u}_t dx,
		\end{align*}
		by integration by parts. Plugging this identity into \eqref{wfe} finishes the proof.
	\end{proof}
	Note that the choice of the functional $\F(\bb{u},\theta)$ is motivated by the need to control both displacement and temperature gradients simultaneously, using Fisher information, which naturally captures the smoothing properties of the heat equation. Its monotonicity-like properties in Lemma $\ref{glestlem}$ will enable us to obtain estimates essential for global-in-time existence. First let us show the following inequality, which is the torus version of \cite[Lemma~A.1]{Fuestin}.
	\begin{lem}\label{inpion}
		There is $C > 0$ (depending only on $d$) such that the inequality
		\begin{align*}
			\|\nabla v\|_{\ld}\leq C\|\Delta v\|_{\ld}
		\end{align*}
		is satisfied for every $v\in \htwo$.
	\end{lem}
	\begin{proof}
		Let us denote $\langle v\rangle=\io v dx$. We calculate
		\begin{align*}
			\io|\nabla v|^2 dx=-\io (v-\langle v\rangle)\Delta vdx.
		\end{align*}
		Here, we use the Schwarz and the Poincar\'e inequality. We get
		\begin{align*}
			\|\nabla v\|^2_{\ld}\leq\|v-\langle v\rangle\|_{\ld}\|\Delta v\|_{\ld}\leq C\|\nabla v\|_{\ld}\|\Delta v\|_{\ld}.
		\end{align*}
		We divide the above inequality by $\|\nabla v\|_{\ld}$ and get the claim.
	\end{proof}
	
	Similarly, the below inequality is the torus version of \cite[Lemma~A.1]{CFHS}
	and can be obtained in the same way; note that in the toroidal settings no boundary terms emerge.
	\begin{lem}\label{theCin}
		Let $w\in C^2(\tn)$ be a positive function, then we have
		\begin{align*}
			\io |\nabla^2w^{\sul}|^2dx\leq C\io w|\nabla^2\log w|^2dx,
		\end{align*}
		where $C=1+\frac{\sqrt d}2+\frac d8$.
	\end{lem}
	
	In the remainder of the section, we derive estimates for $\F(\bb u, \theta)$ both for small and general initial data.
	The following lemma will be used for both cases.
	\begin{lem}\label{F_first_est}
		Let $(\bb{u}, \theta)$ with $\theta > 0$ be a smooth solution of \eqref{new:system}. 
		Then
		\begin{align}\label{F_first_est:est}
			\ddt\F(\bb u, \theta) &\leq -C_1 \|\nabla^2\theta^{\sul}\|_{\ld}^2 + C_2 \F^\frac12(\bb u, \theta) \|\nabla \theta^{\sul}\|_{L^4(\tn)}^2,
		\end{align}
		where $C_1, C_2 > 0$ only depend on $d$, $\mu$ and $\tn$.
	\end{lem}
	\begin{proof}
		By Lemma~\ref{wonderfulformula}, we get
		\begin{align*}%\label{in4estsol}
			\ddt \F(\bb u, \theta) =-\io\theta|\nabla^2\log\theta^{\sul}|^2dx-\frac\mu2\io\frac{|\nabla\theta|^2}\theta\div \bb{u}_tdx \eqqcolon I_1 + I_2.
		\end{align*}
		The term $I_1$ can be estimated by Lemma~\ref{theCin}. It yields
		\begin{align*}%\label{in3estsol}
			I_1\leq -C_1\|\nabla^2\theta^{\sul}\|_{\ld}^2,
		\end{align*}
		while on the other hand by Hölder's inequality,
		\begin{align*}
			I_2
			&\leq   C_2\|\nabla\theta^{\sul}\|_{L^4(\tn)}^2\|\div \bb{u}_t\|_{\ld}
			\leq    C_2 \F^\frac12(\bb u, \theta) \|\nabla\theta^{\sul}\|_{L^4(\tn)}^2.
			\qedhere
		\end{align*}
	\end{proof}
	
	In the next lemma, we present one of the key ingredients of our analysis. It provides the fundamental estimates necessary for proving our main results and will play a crucial role throughout the remainder of the paper.
	\begin{lem}\label{glestlem}
		Let $(\bb{u}, \theta)$ with $\theta > 0$ be a smooth solution of \eqref{new:system}. 
		Then
		\begin{align*}
			\ddt\F\leq C\|\nabla^2\theta^{\sul}\|_{\ld}^2(\F-D),
		\end{align*}
		where $C, D > 0$ only depend on $d$, $\mu$ and $\tn$.
	\end{lem}
	\begin{proof}
		By combining Lemma~\ref{F_first_est} and Lemma~\ref{inpion} and utilizing Young's inequality, we obtain
		\begin{align*}
			\ddt\F
			&\leq -C_1\|\nabla^2\theta^{\sul}\|_{\ld}^2+C_2 \F^\frac12(\bb u, \theta)\|\nabla^2\theta^{\sul}\|_{\ld} \\
			&\leq (\epsilon-C_1)\|\nabla^2\theta^{\sul}\|_{\ld}^2+ \frac{C_2}{\epsilon} \F(\bb u, \theta)\|\nabla^2\theta^{\sul}\|_{\ld}^2.
		\end{align*}
		Taking $\epsilon:=\frac{C_1}2$ and then $C=\frac{2C_2}{C_1}$ and $D=\frac{C_1^2}{4C_2}$, the conclusion follows.
	\end{proof}
	As a direct consequence of the above lemma, we arrive at the following estimate. 
	\begin{cor}[Global estimate for small data]\label{est:sol}
		Let $(\bb{u}, \theta)$ with $\theta > 0$ be a smooth solution of \eqref{new:system}. 
		Then
		\begin{eqnarray*}%%%%%%%%%%%%%
			\mathcal{F}(\bb u_0, \theta_0)\leq D \implies \mathcal{F}(\bb u(\cdot, t), \theta(\cdot, t))\leq \mathcal{F}(\bb u_0, \theta_0)  \quad \text{for all } t\in(0,\infty),
		\end{eqnarray*}%%----------------------------%%
		where $D$ is defined in Lemma~\ref{glestlem}.
	\end{cor}
	
	We now have the following estimate in the general case:
	\begin{lem}[Local estimate for large data]\label{loc:est}
		Let $(\bb{u}, \theta)$ with $\theta > 0$ be a smooth solution of \eqref{new:system}. 
		Then
		\begin{eqnarray*}%%%%%%%%%%%%%
			\mathcal{F}(\bb u(\cdot, t), \theta(\cdot, t))\leq f(t)
		\end{eqnarray*}%%----------------------------%%
		for all $t\in (0,T_{max})$, where $T_{max}>0$ and $f \in C(0, T_{max})$ only depend on $\mathcal{F}(\bb u_0, \theta_0), d, \mu,\tn$,
		while $f\nearrow \infty$ as $t\nearrow T_{max}$.
	\end{lem}
	\begin{proof}
		We utilize the Gagliardo-Nirenberg inequality to get
		\begin{align*}
			\|\nabla\theta^{\sul}\|_{L^4(\tn)}^2\leq C\left(\|\nabla^2\theta^{\sul}\|_{\ld}^{2\eta}\|\nabla\theta^{\sul}\|_{\ld}^{2(1-\eta)}+\|\nabla\theta^{\sul}\|_{\ld}^2\right),
		\end{align*}
		where $\eta=\frac12$ for $d=2$ and $\eta=\frac34$ for $d=3$. We put this estimate into \eqref{F_first_est:est} and obtain
		\begin{align*}
			\ddt\F(\bb u, \theta)
			&\leq -C_1\|\nabla^2\theta^{\sul}\|_{\ld}^2\\
			&\mathrel{\hphantom{\leq}}+C_2 \F^\frac12(\bb u, \theta)\|\nabla^2\theta^{\sul}\|_{\ld}^{2\eta}\|\nabla\theta^{\sul}\|_{\ld}^{2(1-\eta)}+C_2 \F^\frac12(\bb u, \theta)\|\nabla\theta^{\sul}\|_{\ld}^2.
		\end{align*}
		Next, by the Young inequality we have
		\begin{align*}
			\ddt\F(\bb u, \theta)
			&\leq -C_1\|\nabla^2\theta^{\sul}\|_{\ld}^2
			+C_2 \F^\frac{\alpha}{2}(\bb u, \theta) \|\nabla\theta^{\sul}\|_{\ld}^2
			+C_2 \F^\frac12(\bb u, \theta) \|\nabla\theta^{\sul}\|_{\ld}^2\\
			&\leq C(\F^{1+\frac\alpha2}(\bb u, \theta)+\F^{\frac32}(\bb u, \theta))\leq C(\F^{1+\frac\alpha2}(\bb u, \theta)+1),
		\end{align*}
		where $\alpha=2$ for $d=2$ and $\alpha=4$ for $d=3$. The above inequality leads us to the desired conclusion,
		with $f$ being the solution to the ODE initial value problem $f' = C(f^{1+\frac\alpha2}+1)$, $f(0) = \F(\bb u_0, \theta_0)$
		and $T_{max} \in (0, \infty)$ being its maximal existence time.
	\end{proof}

	\section{Existence of unique solutions -- proof of Theorem~\ref{main1}}
	In this section, we construct solutions to \eqref{new:system}. We mimic the construction in \cite{CMT} on the torus $\tn$ by using a half-Galerkin procedure for solving the problem.
	
	\subsection{Approximate problem, approximate solutions and uniform estimates}
	Let $\{\phi_i\}$ be an orthonormal basis of scalar eigenfunctions\footnote{Let us notice that the function $\varphi_1\equiv1$ is also an eigenfunction of $-\Delta$ on $\tn$.} of $-\Delta$ in $\hom$ and denote 
	\begin{eqnarray*}%%%%%%%%
		V_n \coloneqq
		\begin{cases}
			\operatorname{span} \{(\varphi_i,0), (0,\varphi_i)\}_{1\leq i \leq n}, & \text{if } d=2,\\
			\operatorname{span} \{(\varphi_i,0,0), (0,\varphi_i,0), (0,0,\varphi_i)    \}_{1\leq i \leq n}, & \text{if } d=3.
		\end{cases}
	\end{eqnarray*}%%----------------------------%%
	Let us state our approximate problem:
	\begin{defi}\label{apdef}
		We say that $(\bb{u}_n, \theta_n) \in C^1([0,\infty);V_n) \times C^0([0,\infty);\hom)$ is a global solution of an approximate problem if for all $\bb \phi\in V_n$ and $\psi\in \hom$ the following equations are satisfied in $\mathcal{D}'((0, \infty))$:
		\begin{align}
			\frac{d^2}{dt^2}\io \bb{u}_n\boldsymbol\phi dx+\io \nabla \bb{u}_n\cdot\nabla\boldsymbol\phi dx&=\mu\io\theta_n\div\boldsymbol\phi dx,\label{app:mom}\\
			\ddt\io\theta_n\psi dx+\io \nabla\theta_{n}\cdot\nabla\psi dx&=\mu\io\psi\theta_n \div \bb{u}_{n,t}dx. \label{app:heat}
		\end{align}
		Moreover, the approximate initial data are chosen as
		\begin{align*}
			\theta_n(0)=\theta_{0,n},\quad \bb{u}_n(0)=P_{V_n}\bb{u}_0,\quad \bb{u}_{n,t}(0)=P_{V_n}\bb{v}_0,
		\end{align*}
		where $P_{V_n}$ is the orthogonal projection of the corresponding spaces onto $V_n$, while $\theta_{0,n}$ is a strictly positive smooth regularization of $\theta_0$ such that $\theta_{0,n}\to \theta_0$ in $H^1(\tn)$ as $n\to\infty$ and\footnote{For example, this can be done by mollifying $\theta_0$ and adding a positive constant $c_n$ to reduce this norm, with $c_n\searrow 0$ as $n\to \infty$.}
		\begin{eqnarray*}%%%%%%%%%%%%%
			\left\|\frac{\nabla \theta_{0,n}}{\sqrt{\theta_{0,n}}}\right\|_{L^2(T^d)} \leq \left\|\frac{\nabla \theta_0}{\sqrt{\theta_0}}\right\|_{L^2(T^d)}.
		\end{eqnarray*}%%----------------------------%%
	\end{defi}
	We have the following existence result:
	\begin{prop}\label{aproxex}
		For every $n\in\mathbb{N}$, there exists a global solution $(\bb{u}_n,\theta_n)$ of the approximate problem in the sense of Definition \ref{apdef}.
		Moreover, one has $\theta_n\in L_{\loc}^2([0,\infty);\htwo)\cap H_{\loc}^1([0,\infty);\ld)$ and for every $T \in (0, \infty)$ there exist $\hat\theta_n>0$ such that $\theta_n(x,t)\geq\hat\theta_n$ for almost all $(t,x)\in [0,T]\times\tn$. The total energy balance
		\begin{align}
			&\ul\io |\bb{u}_{n,t}|^2(t)dx+\ul\io |\nabla \bb{u}_n|^2(t)dx+\io\theta_n(t) dx\nonumber\\
			&=\ul\io |\bb{v}_n(0)|^2dx+\ul\io |\nabla \bb{u}_n(0)|^2dx+\io\theta_n(0) dx,\label{balance:app}
		\end{align}
		and total dissipation balance
		\begin{eqnarray}%%%%%%%%%%%%%
			&&\ul\io|\bb{u}_{n,t}|^2(t)dx+\ul\io |\nabla \bb{u}_n|^2(t)dx+\io(\theta_n-\log\theta_n)(t) dx + \int_0^t \io |\nabla\log\theta_n|^2 dx d\tau \nonumber\\
			&&=\ul\io |\bb{v}_n(0)|^2dx+\ul\io |\nabla \bb{u}_n(0)|^2dx+\io(\theta_n(0)-\log\theta_n(0)) dx, \label{taul2:app}
		\end{eqnarray}%%----------------------------%%
		hold for all $t\in (0,\infty)$.
	\end{prop}
	\begin{proof}
		For the reader's convenience, let us briefly summarize the main steps. The solutions of \eqref{app:mom} and \eqref{app:heat} are obtained by standard and parabolic ODE theory, respectively, and the solution to the whole system \eqref{app:mom}--\eqref{app:heat} is obtained by Schaefer’s fixed point theorem. The regularity for $\theta_n$ follows by standard $L^2$ estimates for the heat equation, while the $T$-dependent lower bound for $\theta_n$ follows by a maximum principle argument. The identities \eqref{balance:app} and \eqref{taul2:app} follow from Lemma~$\ref{uni:est}$, as the regularity is sufficient to carry out the proof on this approximate level. More details can be found in \cite[Proposition~1 and Lemma~4.2]{CMT}.
	\end{proof}
	
	\begin{rem}\label{approx_smooth}
		The solution $(\bb{u}_n,\theta_n)$ obtained in the previous lemma is regular.
		More precisely, we have $\bb{u}_n \in C^2([0,T]; V_n)$ and $\theta_n\in L^2(0,T; H^2(\tn))$, and since $V_n$ consists of smooth eigenfunctions, one obtains that the source term of the heat equation $\eqref{app:heat}$ is $-\mu \theta_n \div \bb{u}_{n,t} \in L^2(0,T; W^{1,p}(\tn))$ for some $p>3$, so by maximal regularity estimates \cite[Theorem 4.10.7]{Amann}, one has that $\theta_n\in C([0,T]; W^{1,p}(\tn)) \hookrightarrow C([0,T]; C^{0,\alpha}( \tn))$ for some $\alpha\in(0,1)$. Now $-\mu \theta_n \bb{u}_n \in C([0,T]; C^{0,\alpha}(\tn))$ so $\theta_n$ is a classical solution to $\eqref{app:heat}$ with $\theta_n \in C([0,T]; C^{2,\alpha}(\tn))$ and $\theta_{n,t}\in ([0,T]; C^{0,\alpha}(\tn))$ (see \cite[Theorem 5.1.2]{Lunardi}).
	\end{rem}
	
	By Proposition \ref{aproxex}, we directly obtain the uniform estimate summarized in the following corollary.
	\begin{cor}\label{uni:est:apr}
		The approximate solutions satisfy the following uniform estimates:
		\begin{align}%\label{uni:est}%%%%%%%%%%%%%
			\|\bb{u}_{n,t}\|_{L^\infty(0,\infty; L^2(\mathbb{T}^d) )} \leq C,\nonumber \\
			\|\nabla\bb{u}_{n}\|_{L^\infty(0,\infty; L^2(\mathbb{T}^d) )} \leq C,\nonumber\\
			\|\theta_n\|_{L^\infty(0,\infty; L^1(\mathbb{T}^d) )} \leq C,\nonumber \\
			\|\log\theta_n\|_{L^\infty(0,\infty; L^1(\mathbb{T}^d) )} \leq C,\label{uni:est4}\\
			\|\nabla\log\theta_n\|_{L^2(0,\infty; L^2(\mathbb{T}^d) )} \leq C, \nonumber
		\end{align}%%----------------------------%%
		where $C$ does not depend on $n$.
	\end{cor}
	
	\subsubsection{Global estimates for small data}
	In this section, we establish global estimates for small initial data, ensuring uniform bounds that will be crucial for proving global existence.
	
	Let us first formulate an auxiliary regularity fact.
	
	\begin{prop}\label{pomocnicze}
		For all $M > 0$ there exists $K = K (M, d, \tn) > 0$ such that for all $T \in (0, \infty]$, all $g \colon (0, T) \times \tn \to (0, \infty)$ with $\sqrt{g}\in L^\infty(0,T;H^1(\tn))$ and all $h \in W^{1,\infty}(0,T;L^2(\tn))$ satisfiying
		\[
		\|\sqrt{g}\|_{L^\infty(0,T;H^1(\tn))} \le M
		\quad \text{and} \quad
		\|h\|_{W^{1,\infty}(0,T;L^2(\tn))} \le M,
		\] 
		the functions $g$ and $g h_t$ belong to $L^\infty(0,T; W^{1,3/2}(\tn))$ and $L^{\infty}(0,T;L^{6/5}(\tn))$, respectively, and fulfill
		\[
		\|g\|_{L^\infty(0,T; W^{1,3/2}(\tn))} \le K
		\quad \text{and} \quad
		\|g h_t\|_{L^{\infty}(0,T;L^{6/5}(\tn))} \le K.
		\]
	\end{prop}
	\begin{proof}
		First, we notice
		\begin{eqnarray*}%%%%%%%%%%%%%
			&&\|g\|_{L^\infty(0,T; L^{3}(\tn))} = \|\sqrt{g}\|_{L^\infty(0,T; L^{6}(\tn))}^2 \leq C\|\sqrt{g}\|_{L^\infty(0,T; H^1(\tn))}^2 \leq CM^2,\\
			&&\|\nabla g\|_{L^\infty(0,T; L^{3/2}(\tn))}=
			\left\| \sqrt{g} \frac{\nabla g}{\sqrt{g}}\right\|_{L^\infty(0,T; L^{3/2}(\tn))} \leq  \|\sqrt{g}\|_{L^\infty(0,T; L^{6}(\tn))}  \left\|\frac{\nabla g}{\sqrt{g}}\right\|_{L^\infty(0,T; L^{2}(\tn))}\leq CM^2,
		\end{eqnarray*}%%----------------------------%%
		so also $\|g\|_{L^\infty(0,T; W^{1,3/2}(\tn))} \le CM^2$.
		Since moreover
		\[
		\|gh_t\|_{L^\infty(0, T; L^{6/5}(\tn))} \le \|g\|_{L^\infty(0, T; L^{3}(\tn))} \|h_t\|_{L^\infty(0, T; L^{2}(\tn))}
		\]
		by Hölder's inequality, we also obtain the second estimate.
	\end{proof}
	
	\begin{lem}\label{glob:est}
		Assume that $(\bb{u}_n,\theta_n)$ is a global solution to an approximate problem in the sense of Definition \ref{apdef} and that the initial values satisfy the inequality
		\begin{align}\label{globest:initcond}
			\initcond,
		\end{align}
		where $D$ is defined in Lemma \ref{est:sol}. Then for any $T>0$:
		\begin{align}
			\|\nabla\bb{u}_{n,t}\|_{L^{\infty}(0,T;\ld)}&\leq C, \label{t1glest}\\
			\|\Delta\bb{u}_{n}\|_{L^{\infty}(0,T;\ld)}&\leq C,\nonumber\\
			\|\sqrt{\theta_n}\|_{L^\infty(0,T; H^1(\tn))}&\leq C, \label{t4glest} \\
			\|\theta_{n}\|_{L^2(0,T; H^{2}(\tn))} +\|\nabla\theta_n\|_{C([0,T]; L^{2}(\tn))} +\|\theta_{n,t}\|_{L^2(0,T; L^2(\tn))} &\leq \tilde{C}(T), \nonumber \\
			\|\bb{u}_{n,tt}\|_{L^{\infty}(0,T;\ld)}&\leq \tilde{C}(T),\label{t3glest}
		\end{align}
		where $C,m,M$ do not depend on $n$ or $T$ and $\tilde{C}(t)$ does not depend on $n$ and $\tilde{C}(t)\nearrow \infty$ as $t\to\infty$.
	\end{lem}
	\begin{proof}
		By Remark~\ref{approx_smooth}, all necessary calculation concerning the identity in Lemma~\ref{wonderfulformula} and the inequality in Lemma~\ref{glestlem} hold at the level of approximate solutions.
		(We note that in the proof of Lemma~\ref{wonderfulformula}, the first equation is only tested with $-\Delta \bb u_{n, t} \in V_n$.)
		This leads to the following uniform bounds which follow from Corollary~\ref{est:sol}:
		\begin{eqnarray*}%%%%%%%%%%%%%
			\|\nabla \bb{u}_{n,t}\|_{L^\infty(0,T;L^2(T^d))} \leq C, \quad 
			\|\Delta \bb{u}_n\|_{L^\infty(0,T;L^2(T^d))} \leq C, \quad 
			\left\|\frac{\nabla \theta_n}{\sqrt{\theta_n}}\right\|_{L^\infty(0,T;L^2(T^d))} \leq C.
		\end{eqnarray*}%%----------------------------%%
		We are now in a position to apply Proposition~\ref{pomocnicze} with $g\coloneqq\theta_n$ and $h\coloneqq\div \bb u_n$.  Consequently,
		we have that $\theta_n \div(\bb u_{n,t})$ is uniformly bounded in $L^\infty(0,T; L^{6/5}(\tn)) \hookrightarrow L^2(0,T; L^{6/5}(\tn))$. Therefore, by the maximal regularity estimates (cf.\ \cite[Theorem 4.10.7]{Amann}), one has 
		\begin{eqnarray*}%%%%%%%%%%%%%
			\|\theta_n\|_{H^1(0,T; L^{6/5}(\tn))} + \|\theta_n\|_{L^2(0,T; W^{2,6/5}(\tn))} \leq \tilde{C}_1(\|\theta_{0,n}\|_{W^{1,6/5}(\tn)}+ \|\theta_n \div \bb{u}_{n,t}\|_{L^2(0,T; L^{6/5}(\tn))})
		\end{eqnarray*}%%----------------------------%%
		where $\tilde{C}_1$ depends on $T,\tn,\mu$. This in particular implies that $\theta_n \div \bb{u}_{n,t}$ can be bounded in a better space, so we can again (successively) apply the same maximal regularity arguments to increase the regularity
		\begin{align*}%%%%%%%%%%%%%
			\|\theta_n\|_{H^1(0,T; L^{2}(\tn))}& + \|\nabla\theta_n\|_{C([0,T]; L^{2}(\tn))}+\|\theta_n\|_{L^2(0,T; H^{2}(\tn))} \\
			&\leq \tilde{C}(\|\theta_{0,n}\|_{H^{1}(\tn)}+ \|\theta_n \div \bb{u}_{n,t}\|_{L^2(0,T; L^{2}(\tn))}),
		\end{align*}%%----------------------------%%
		where $\tilde{C}$ depends on $T,\tn,\mu$. Since $\bb{u}_{n,tt} = \Delta\bb{u}_n -\mu P_{V_n}\nabla\theta_n$, the estimate $\eqref{t3glest}$ follows.
		\iffalse
		This also implies $\eqref{nabla:thetan}$ due to
		\begin{eqnarray*}%%%%%%%%%%%%%
			\left\|\nabla \theta_n\right\|_{L^{\infty}(0,T;\ld)}\leq M\left\|\frac{\nabla\theta_n}{\sqrt{\theta_n}}\right\|_{L^{\infty}(0,T;\ld)} \leq C
		\end{eqnarray*}%%----------------------------%%
		and since \fi
	\end{proof}
	
	\subsubsection{Local estimates for large data}
	
	In this section, we establish local-in-time estimates for solutions corresponding to large initial data.
	Unlike in the previous section, we now derive estimates that remain valid also for initial data not fulfilling \eqref{globest:initcond} but only hold up to some finite time \(T_{\max}\). 
	
	\begin{lem}\label{l:est}
		Assume that $(\bb{u}_n,\theta_n)$ is a global solution to an approximate problem in the sense of Definition \ref{apdef}.
		Then, independently of $n$, there exist $T_{max}>0$ and $f \in C^0([0, T_{max}); (0, \infty))$ with $f(t) \nearrow \infty$ for $t \nearrow T_{max}$ such that for all $n \in \mathbb N$ and all $T<T_{max}$,
		\begin{align*}
			\|\nabla\bb{u}_{n,t}\|_{L^{\infty}(0,T;\ld)}&\leq f(T),\\
			\|\Delta\bb{u}_{n}\|_{L^{\infty}(0,T;\ld)}&\leq f(T),\\
			\left\|\sqrt{ \theta_n}\right\|_{L^{\infty}(0,T;\hom)}&\leq f(T),  \\
			\|\theta_{n}\|_{L^2(0,T; H^{2}(\tn))} +\|\nabla\theta_n\|_{C([0,T]; L^{2}(\tn))} +\|\theta_{n,t}\|_{L^2(0,T; L^2(\tn))}  &\leq f(T), \\
			\|\bb{u}_{n,tt}\|_{L^{\infty}(0,T;\ld)}&\leq f(T). 
		\end{align*}
	\end{lem}
	\begin{proof}
		In this case, the proof is the same as in global case, but instead of Corollary~\ref{est:sol}, we rely on Lemma~\ref{loc:est}.
	\end{proof}
	
	\subsection{Passing to the limit \texorpdfstring{$n\to\infty$}{n to infty}}
	In this section, we utilize the estimates obtained in the previous section for the approximate sequence $(\bb{u}_n,\theta_n)$, and then pass to the limit $n\to\infty$ to obtain the solution to the original problem.
	
	\begin{lem}\label{local:est}
		Let $(\bb{u}_n,\theta_n)$ be a global solution to the problem \eqref{app:mom}--\eqref{app:heat} obtained in Proposition~\ref{aproxex}. Then, there exists a solution $(\bb{u},\theta)$ in the sense of Definition $\ref{soldef}$ such that 
		\begin{equation}\label{weakex:alt}
			\begin{alignedat}{2}
				\bb{u}_n&\stackrel *{\rightharpoonup}\bb{u}           &&\quad \textrm{in }L_{\loc}^{\infty}([0,T);\htwo),\\
				\bb{u}_{n,t}&\stackrel *{\rightharpoonup}\bb{u}_t     &&\quad \textrm{in }L_{\loc}^{\infty}([0,T);\hom),\\
				\bb{u}_{n,tt}&\stackrel *{\rightharpoonup}\bb{u}_{tt} &&\quad \textrm{in }L_{\loc}^{\infty}([0,T);\ld),\\
				\theta_{n}&\stackrel{}{\rightharpoonup}\theta         &&\quad \textrm{in }H_{\loc}^1([0,T);\ld),\\
				\theta_{n}&\stackrel{}{\rightharpoonup}\theta         &&\quad \textrm{in }L_{\loc}^2([0,T);\htwo), \\
				\theta_n& \to \theta                                  &&\quad \textrm{in }L_{\loc}^{2}([0,T);\ld), \\
				\log\theta_n& \to \log\theta                            &&\quad \textrm{in }L_{\loc}^{2}([0,T);\ld) \\
			\end{alignedat}
		\end{equation}
		along some subsequence,
		where
		\begin{eqnarray*}%%%%%%%%%%%%%
			&T = \infty& \text{ if } \quad \initcond, \\
			&T = T_{max} & \text{ otherwise},
		\end{eqnarray*}%%----------------------------%%
		where $D$ and $T_{max}$ are defined in Lemma~\ref{glestlem} and Lemma~\ref{loc:est}, respectively.
		
	\end{lem}
	\begin{proof}
		Without indicating the switch to subsequences, we observe that convergences \eqref{weakex:alt}$_1$--\eqref{weakex:alt}$_5$ follow directly by the uniform bounds obtained in Corollary~\ref{uni:est:apr}, Lemma~\ref{glob:est}, Lemma~\ref{l:est}, while the strong convergence \eqref{weakex:alt}$_6$ is a consequence of compact imbedding of $H_{\loc}^1([0,T);\ld)\cap L_{\loc}^2([0,T);\htwo)$ into $L_{\loc}^{2}([0,T);\ld)$.
		
		Next, note that since $\log\theta_n \in L^\infty(0,T; L^1(\tn))$ and $\nabla\log\theta_n \in  L^2((0,T) \times \tn)$ by \eqref{taul2:app}, we have that $\log\theta_n \in L_{\loc}^2([0,T); H^1(\tn))$. Now, since $\theta_n\geq \hat{\theta}_n>0$, we can divide the heat equation $\eqref{app:heat}$ by $\theta_n$ to arrive at 
		\begin{eqnarray*}%%%%%%%%%%%%%
			\log\theta_{n,t} = \Delta\log\theta_n + |\nabla\log\theta_n|^2 - \mu \div\bb{u}_{n,t} \in L^1(0,T;W^{-1,p}(\mathbb{T}^d))
		\end{eqnarray*}%%----------------------------%%
		for $p\in (1,3/2)$, so by the Aubin--Lions lemma (see \cite{Boyer}), we have $\log\theta_n \to \tau$ in $L_{\loc}^{2}([0,T);\ld)$ and a.e.\ on $(0,T)\times \mathbb{T}^d$. In order to show $\tau=\log\theta$, observe that for a.a.\ $(t,x)\in (0,T)\times \mathbb{T}^d$ we have $\lim_{n\to\infty}\log\theta_n(t,x) = \tau(t,x)$ and $\lim_{n\to\infty}\theta_n(t,x) = \theta(t,x)$, so fixing such $(t,x)$ and calculating
		\begin{eqnarray*}%%%%%%%%%%%%%
			\theta(t,x)=\lim_{n\to\infty} \theta_n(t,x)=\lim_{n\to\infty} e^{\log\theta_n(t,x)}= e^{\lim_{n\to\infty} \log\theta_n(t,x)} = e^{\tau(t,x)},
		\end{eqnarray*}%%----------------------------%%
		$\eqref{weakex:alt}_7$  follows.
	\end{proof}

	\subsection{Uniqueness of solutions}
	The following result holds for both global-in-time and local-in-time solutions, as its validity depends solely on the regularity class of the solution.
	In particular, it entails the uniqueness claim in Theorem~\ref{main1} and hence together with Lemma~\ref{local:est} proves Theorem~\ref{main1}.
	\begin{lem}
		Let us fix the time of existence $T$ of solutions to \eqref{new:system} in the sense of Definition \ref{soldef}. Then, the solution satisfying \eqref{new:system} on $(0,T)$ in the regularity class described in Definition \ref{soldef} is unique.
	\end{lem}
	\begin{proof}
		Let $(\bb{u}_1,\theta_1)$ and $(\bb{u}_2,\theta_2)$ be solutions of \eqref{new:system} in the sense of Definition \ref{soldef} and let us denote $\bb u:=\bb{u}_1-\bb{u}_2$ and $\theta:=\theta_1-\theta_2$. Let us subtract the equations for $(\bb{u}_2,\theta_2)$ from the equations for $(\bb{u}_1,\theta_1)$. We obtain
		\begin{align*}
			\begin{split}
				\bb{u}_{tt}-\Delta \bb{u} &=   -\mu\nabla\theta,\\
				\theta_t - \Delta \theta &=- \mu\theta_1 \div\bb{u}_{1,t}+\mu\theta_2\div \bb{u}_{2,t}.
			\end{split}
		\end{align*}
		In order to work with balanced estimates, let us apply the divergence to the first equation (or more precisely multiply with $\nabla \varphi$ and integrate over $(0,T)\times \tn$) to obtain the new system
		\begin{eqnarray}\label{equniq}%%%%%%%%%%%%%
			\begin{cases}  \int_0^T \int_{\tn} m_{t}\varphi_t-\int_0^T \int_{\tn} \nabla m \cdot \nabla \varphi&=   \mu \int_0^T \int_{\tn} \Delta\theta \varphi+ \int_{\tn} m_t(T) \varphi(T),\\
				\theta_t - \Delta \theta &=- \mu\theta_1 m_{1,t}+\mu\theta_2 m_{2,t},
			\end{cases}
		\end{eqnarray}%%----------------------------%%
		for all $\varphi \in C^1([0,T]\times \tn)$, where we have denoted $m=\div\bb{u}$ and $m_i=\div \bb{u}_i$. Note that here we used
		\begin{eqnarray*}
			\int_0^T \int_{\tn} \Delta \bb{u} \cdot \nabla \varphi = \int_0^T \int_{\tn} \nabla (\nabla \cdot \bb{u}) \cdot \nabla \varphi - \underbrace{\int_0^T \int_{\tn} \curl^2 \bb{u} \cdot \nabla \varphi}_{=0}.
		\end{eqnarray*}
		Since $\eqref{equniq}_1$ implies that $m_{tt}\in L^2(0,T; H^{-1}(\tn))$, we can test $\eqref{equniq}_1$ with $P_{V_k}m_t$ by the density argument and pass to the limit $k\to \infty$ to obtain
		\begin{align*}
			&\frac12\io |m_t(t)|^2dx +\frac12\io |\nabla m(t)|^2dx \leq -\mu \int_0^t\io m_t\Delta\theta dxd\tau\nonumber\\
			&\leq \varepsilon \|\Delta\theta\|_{L^2((0,T)\times \tn)}^2 + C \int_0^t \int_{\tn} |m_t|^2dxdt.
		\end{align*}
		Next, multiplying $\eqref{equniq}_2$ with $\theta$ and then with $-\Delta \theta$, integrating over $(0,t)\times\tn$ and summing up gives us
		\begin{eqnarray*}%%%%%%%%%%%%%
			&&\frac12 \int_{\tn} \theta^2(t) + \|\nabla\theta\|_{L^2((0,T)\times \tn)}^2 + \frac12 \int_{\tn} |\nabla\theta|^2(t) + \|\Delta\theta\|_{L^2((0,T)\times \tn)}^2 \\
			&&\leq C\int_0^t \int_{\tn} (\theta_1 m_{1,t} - \theta_2 m_{2,t})^2 + \varepsilon\left( \|\theta\|_{L^2((0,T)\times \tn)}^2 + \|\Delta\theta\|_{L^2((0,T)\times \tn)}^2\right) 
		\end{eqnarray*}%%----------------------------%%
		Now,
		\begin{eqnarray*}%%%%%%%%%%%%%
			&&\int_0^t \int_{\tn} (\theta_1 m_{1,t} - \theta_2 m_{2,t})^2 =\int_0^t \int_{\tn} (\theta_1 m_{t} - \theta m_{2,t})^2   \\
			&&\leq 2\int_0^t \int_{\tn} \theta_1^2 m_t^2 + 2\int_0^t \int_{\tn} \theta^2 m_{2,t}^2 \\
			&& \leq 2\int_0^t \|\theta_1\|_{L^\infty(\tn)}^2  \|m_t\|_{L^2(\tn)}^2 +  2\int_0^t \underbrace{\|m_{2,t}\|_{L^2(\tn))}^2}_{\leq C}\|\theta\|_{L^\infty(\tn)}^2 \\
			&&\leq C\int_0^t \|\theta_1\|_{H^2(\tn)}^2  \|m_t\|_{L^2(\tn)}^2 + C \int_0^t \|\theta\|_{L^2(\tn)}^2+\varepsilon\left(\|\theta\|_{L^2((0,T)\times \tn)}^2 + \|\Delta\theta\|_{L^2((0,T)\times \tn)}^2 \right)
		\end{eqnarray*}%%----------------------------%%
		by imbedding, where we bounded 
		\begin{eqnarray*}%%%%%%%%%%%%%
			&&\|\theta\|_{L^\infty(\tn)}\leq C\|\theta\|_{H^{7/4}(\tn)} \leq C\|\theta\|_{L^2(\tn)}^{\frac18}\|\theta\|_{H^2(\tn)}^{\frac78}\leq C\|\theta\|_{L^2(\tn)} + \varepsilon\|\theta\|_{H^2(\tn)} \\
			&&\leq C\|\theta\|_{L^2(\tn)} + \varepsilon\left(\|\theta\|_{L^2(\tn)} + \|\Delta\theta\|_{L^2(\tn)}\right)
		\end{eqnarray*}%%----------------------------%%
		by imbedding, interpolation and Young's inequality. Denoting 
		\begin{eqnarray*}%%%%%%%%%%%%%
			f(t):=\frac12\io m_t^2(t)dx +\frac12\io |\nabla m(t)|^2dx + \frac12 \int_{\tn} \theta^2(t)
		\end{eqnarray*}%%----------------------------%%
		and combining the above estimates, we obtain
		\begin{eqnarray*}%%%%%%%%%%%%%
			f(t) \leq \int_0^t h(s)f(s)ds
		\end{eqnarray*}%%----------------------------%%
		where
		\begin{eqnarray*}%%%%%%%%%%%%%
			h(s):=C+C\|\theta_1\|_{H^2(\tn)}^2 \in L^1(0,T)
		\end{eqnarray*}%%----------------------------%%
		so $f(t)=0$ by Gronwall's inequality. Consequently, $\theta = 0$ and now $\bb{u}$ solves a homogeneous wave equation with zero initial data which implies that $\bb{u}=0$, and the proof is finished.
	\end{proof}
	
	\section{Upper and lower bounds of temperature -- proof of Theorem $\ref{main2}$}\label{sec:upper_lower_bdd}
	
	We prove the following results using the Moser--Alikakos iteration, see \cite{alikakos, Moser}. In \cite{BC2} we developed a Moser-type estimate in 1D independently on initial data, however in 1D the Sobolev embeddings are much stronger. Let us also mention that in 1D a Moser--Alikakos type iteration was already used in the context of thermoelasticity in \cite{DH}.
	
	\begin{lem}\label{moser}
		Let $M > 0$.
		Then there exists $K = K(d, \tn, M) > 0$
		such that for all $f \in C^0((0, \infty) \times \tn)$ and all $z \in C^{1, 2}([0, \infty) \times \tn)$
		satisfying
		\begin{equation}\label{moser:pdi}
			z_t \le \Delta z + (z^++1) f, \quad \text{ in } (0, T)\times\tn, \\
		\end{equation}
		and
		\begin{align*}
			\|f\|_{L^\infty(0, \infty; \ld)} \le M, \quad
			\|z^+(0)\|_{L^\infty(\tn)} \le M
			\quad \text{and} \quad
			\|z^+\|_{L^\infty(0, \infty; L^1(\tn))} \le M,
		\end{align*}
		where $z^+\coloneqq \max\{z, 0\}$,
		the estimate
		\begin{align*}
			z \le K
			\quad \text{holds in $(0, \infty) \times \tn$}.
		\end{align*}
	\end{lem}
	\begin{proof}
		For $n\in\n_0$, we multiply \eqref{moser:pdi} by $(z^+)^{2^{n+1}-1}$ and integrate over $\tn$. It yields
		\begin{align}\label{in1supth}
			&\frac1{2^{n+1}}\ddt\io (z^+)^{2^{n+1}}dx+\left(\frac1{2^{n-1}}-\frac1{4^n}\right)\io|\nabla((z^+)^{2^n})|^2dx \nonumber \\
			&\leq \io ((z^+)^{2^{n+1}} + (z^+)^{2^{n+1}-1}) f dx 
			\leq M (1 + 2\|(z^+)^{2^n}\|_{L^4(\tn)}^2).
		\end{align}
		Here, we use the Gagliardo-Nirenberg inequality
		\begin{align*}
			\|\varphi\|_{L^4(\tn)}^2\leq C\|\nabla \varphi\|_{\ld}^{2\eta}\|\varphi\|_{L^1(\tn)}^{2(1-\eta)}+C\|\varphi\|^2_{L^1(\tn)}
			\quad \text{for all $\varphi \in \hom$}
		\end{align*}
		where $\eta=\frac34$ for $d=2$ and $\eta=\frac9{10}$ for $d=3$. We apply the Young inequality to the above formula and arrive at
		\begin{align}\label{moser_gni}
			\|\varphi\|_{L^4(\tn)}^2\leq \epsilon\|\nabla \varphi\|_{\ld}^{2}+C_\epsilon \epsilon^{-\alpha}\|\varphi\|_{L^1(\tn)}^2
			\quad \text{for all $\varphi \in \hom$}
		\end{align}
		where $\alpha=3$ for $d=2$ and $\alpha=9$ for $d=3$. Let us denote $a_n\coloneqq\frac1{2^{n}}-\frac1{2\cdot4^n}$.
		We choose $\epsilon=\frac{a_n}{2M}$ and $\varphi = (z^+)^{2^n}$. Plugging the above inequality into \eqref{in1supth}, we get
		\begin{align*}
			\frac1{2^{n+1}}\ddt\io (z^+)^{2^{n+1}}dx+a_n\io|\nabla((z^+)^{2^n})|^2dx\leq C (1 + a_n^{-\alpha}\|(z^+)^{2^n}\|_{L^1(\tn)}^2).
		\end{align*}
		Since $a_n > 2^{-n}$, it leads us to
		\begin{align*}
			\ddt\io (z^+)^{2^{n+1}}dx+\io|\nabla((z^+)^{2^n})|^2dx\leq C 2^{n+\alpha n}(1 + \|(z^+)^{2^n}\|_{L^1(\tn)}^2).
		\end{align*}
		Applying \eqref{moser_gni} once more yields
		\begin{align*}
			\ddt \|z^+\|_{L^{2^{n+1}}(\tn)}^{2^{n+1}}+C_1\|z^+\|_{L^{2^{n+1}}(\tn)}^{2^{n+1}}\leq C_2 2^{n+\alpha n}(1+\|z^+\|_{L^{2^n}(\tn)}^{2^{n+1}}).
		\end{align*}
		Solving the above differential inequality we get
		\begin{align*}
			\|z^+\|_{L^{\infty}(0,T;L^{2^{n+1}}(\tn))}^{2^{n+1}}
			&\leq 2^{n+\alpha n}C(1+\|z^+\|_{L^{\infty}(0,T;L^{2^{n}}(\tn))}^{2^{n+1}})
			+ \|z_0^+\|_{L^{2^{n+1}}(\tn)}^{2^{n+1}}\\
			&\leq 2^{n+\alpha n}C(1+\|z^+\|_{L^{\infty}(0,T;L^{2^{n}}(\tn))}^{2^{n+1}})
			+ \|z_0^+\|_{L^{\infty}(\tn)}^{2^{n+1}},
		\end{align*}
		where $z_0 \coloneqq z(0)$.
		
		Let us denote $A_n \coloneqq \max\{1,\|z^+\|_{L^{\infty}(0,T;L^{2^{n}}(\tn))},\|z_0^+\|_{L^{\infty}(\tn)}\}$ for $n \in \mathbb N_0$.
		Then, the above inequality can be rewritten as follows:
		\begin{align*}
			A_{n+1}^{2^{n+1}}\leq 2^{n+\alpha n}C + 2^{n+\alpha n}C A_n^{2^{n+1}} + 2^{n+\alpha n}C \le 2^{n+\alpha n}C A_n^{2n+1}
		\end{align*}
		It gives us
		\begin{align*}
			A_{n+1}\leq C^{2^{-(n+1)}}2^{(n+\alpha n)2^{-(n+1)}}A_n
		\end{align*}
		for all $n\in\n$. Solving this recursive inequality yields
		\begin{align*}
			A_n\leq \prod_{k=2}^n C^{2^{-k}}2^{(1+\alpha)(k-1)2^{-k}}A_1\leq A_1C^{\sum_{k=2}^{\infty}2^{-k}}2^{(1+\alpha)\sum_{k=2}^{\infty} (k-1)2^{-k}}.
		\end{align*}
		Because both series on the right-hand side of the above inequality are convergent and as $A_1 \le \max\{1, M\}$ by assumption, we get
		\begin{align*}
			\|z^+\|_{L^{\infty}(0,T;L^{2^{n+1}}(\tn))}\leq C.
		\end{align*}
		Passing to the limit $n\to\infty$, the conclusion follows (see \cite[Theorem~2.14]{adams}).
	\end{proof}
	
	Next, we apply Lemma~\ref{moser} to $z = \theta$ and $z = -\log \theta$ to obtain upper and lower bounds for the temperature.
	\begin{lem}\label{upbofortemp}
		Let $(\bb u, \theta)$ be a global solution obtained in Theorem~\ref{main1} and assume that $\theta_0\in L^\infty(\mathbb{T}^d)$.
		Then there is $C > 0$ such that
		\begin{align*}
			\|\theta\|_{L^{\infty}((0,T)\times\tn)}\leq C.
		\end{align*}
	\end{lem}
	\begin{proof}
		For all $n \in \mathbb N$, the solution $(\bb u_n, \theta_n)$ of the approximate problem obtained in Proposition~\ref{aproxex}
		fulfills \eqref{moser:pdi} with $z \coloneqq \theta_n$ and $f \coloneqq \mu |\div \bb u_{n, t}|$.
		Due to \eqref{t1glest}, \eqref{t4glest} and the assumed boundedness of $\theta_0$,
		we may apply Lemma~\ref{moser} to obtain $\theta_n \le C$ for some $C > 0$ not depending on $n$.
		Since $\theta_n \to \theta$ a.e.\ along some subsequence by Lemma~\ref{local:est} and as $\theta \ge 0$ a.e.,
		this entails the claim.
	\end{proof}
	
	\begin{lem}\label{thlowbotemp}
		Let $(\bb u, \theta)$ be a global solution obtained in Theorem~\ref{main1} and assume that $\theta_0 \ge \underline \theta$ for some $\underline \theta > 0$.
		Then there is $C > 0$ such that
		\begin{align*}
			\theta \ge C \quad \text{a.e.\ on $(0, \infty) \times \tn$}.
		\end{align*}
	\end{lem}
	\begin{proof}
		Again, we denote the solution of the approximate problems obtained in Proposition~\ref{aproxex} by $(\bb u_n, \theta_n)$, $n \in \mathbb N$.
		Since $\theta_n$ is positive by Proposition~\ref{aproxex},
		$z \coloneqq -\log \theta_n$ fulfills $z_t = \Delta z - |\nabla z|^2 + \mu \div \bb u_{n, t}$ in $(0, \infty) \times \tn$
		and hence \eqref{moser:pdi} for $f \coloneqq \mu |\div \bb u_{n, t}|$.
		By \eqref{t1glest}, \eqref{uni:est4}, the assumed lower bound for $\theta_0$ and Lemma~\ref{moser},
		we conclude $-\log \theta_n \le C$ and hence $\theta_n \ge e^{-C}$ for some $n$-independent $C$.
		Thanks to Lemma~\ref{local:est}, the bound carries over to $\theta$.
	\end{proof}
	
	Finally, we can directly conclude Theorem~\ref{main2} from the previous two lemmata.
	
	\section{Time-asymptotics of solutions -- proof of Theorem $\ref{main3}$}\label{czwarta}
	
	In this section, we analyze the time-asymptotic behavior of the obtained solutions. We begin by recalling that Theorem \ref{main1} guarantees the existence of global solutions, provided that the initial data satisfy
	\begin{align}\label{initcondasymp}
		\initcond
	\end{align}
	as well as upper and lower bounds for the initial temperature.
	
	A key observation is that the displacement naturally decomposes into two components: one that exhibits persistent oscillations and another that flattens out due to heat dissipation. Specifically, the divergence-free part of the displacement remains oscillatory, while the curl-free part is dampened by thermal effects. This decomposition arises naturally from the Helmholtz decomposition, which allows us to separate the displacement equation into two distinct parts: a homogeneous wave equation governing the divergence-free component and a strongly coupled system describing the evolution of the curl-free component under the influence of heat propagation.
	
	Ultimately, the curl-free component behaves analogously to a vibrating string in the one-dimensional thermoelasticity model (see \cite{BC2}), while the primary asymptotic distinction in higher dimensions arises from the divergence-free component, which retains its oscillatory nature.

	\subsection{Decomposition of solutions}\label{Sec:Helmholtz}
	In this subsection, we recall the Helmholtz decomposition on the torus. Although both the statement and its proof can be found in the literature, we provide them here for completeness.
	
	\begin{tw}\label{thdecomp}
		Let us assume that $\bb v\in H^k(\tn)$ is a vector field, where $k\in\n\cup\{0\}$. Then, $\bb v$ can be decomposed uniquely as
		\begin{align*}
			\bb v=H\bb v+\hp\bb  v,
		\end{align*}
		where $H\bb v, \hp \bb v\in H^k(\tn)$ satisfy
		\begin{equation*}
			\div H\bb v=0
		\end{equation*}
		and there exists function $\phi\in H^{k+1}(\tn)$ such that
		\begin{align*}
			\io \phi\, dx=0\quad\textrm{ and }\quad\hp\bb v=\nabla\phi.
		\end{align*}
		Moreover, $H\bb v$ and $\hp\bb v$ are orthogonal in $L^2(\tn)$
		and there exists $C = C(d, k) > 0$ such that
		\begin{equation*}
			\|H\bb v\|_{H^k(\tn)} \le C \|\bb v\|_{H^k(\tn)} \quad \text{and} \quad \|\hp \bb v\|_{H^K(\tn)} \le C \|\bb v\|_{H^K(\tn)}.
		\end{equation*}
	\end{tw}
	\begin{proof}
		Let $\phi$ be the unique weak solution of the following problem
		\begin{align*}
			\Delta\phi=\div \bb v\quad\textrm{ and }\quad\io \phi\, dx=0.
		\end{align*}
		Due to elliptic regularity theory (see \cite[Section 6.3]{evans}), we have $\phi\in H^{k+1}(\tn)$ and there exists $c_1 > 0$ depending only on $d$ and $k$ such that $\|\phi\|_{H^{k+1}(\tn)} \le c_1 \|\bb v\|_{H^k(\tn)}$.
		We define
		\begin{align*}
			H\bb v:=\bb v-\nabla\phi,\quad \hp \bb v:=\nabla\phi.
		\end{align*}
		Then, the defined functions verify the conditions in the claim. In particular,
		\[
		\io H\bb v \nabla \phi dx= -\io \div H\bb v \phi dx=0.
		\qedhere
		\]
	\end{proof}
	Next, let us apply the Helmholtz decomposition to \eqref{new:system}. Namely, we will show that the problem divides into two different ones.
	\begin{tw}\label{twdecompsol}
		Let $(\bb u, \theta)$ be a global solution of \eqref{new:system}. Then, $\boldsymbol\nu:=H\bb u$ solves the following problem
		\begin{equation}\label{eqdecnondiv}
			\begin{cases}
				{\boldsymbol\nu}_{tt}-\Delta \boldsymbol\nu =0,& \quad \text{ in } (0, \infty)\times \tn,\\
				{\boldsymbol\nu}(0,\cdot)={\boldsymbol\nu}_0,~ {\boldsymbol\nu}_t(0,\cdot) = \tilde{\boldsymbol\nu}_0,&
			\end{cases}
		\end{equation}
		where $\boldsymbol\nu_0:=H\bb{u}_0$ and $\tilde{\boldsymbol\nu}_0:=H\bb{v}_0$. On the other hand, $(\boldsymbol\chi, \theta)$, where $\boldsymbol\chi:=\hp \bb u$ is a solution of the following system
		\begin{equation}\label{eqdecnoncurl}
			\begin{cases}
				{\boldsymbol\chi}_{tt}-\Delta \boldsymbol\chi =   -\mu\nabla\theta ,& \quad \text{ in } (0,\infty)\times \tn,\\
				\theta_t - \Delta \theta =- \mu\theta \div{\boldsymbol\chi}_t,& \quad \text{ in }  (0,\infty)\times\tn,\\
				{\boldsymbol\chi}(0,\cdot)={\boldsymbol\chi}_0,~ {\boldsymbol\chi}_t(0,\cdot) = \tilde{\boldsymbol\chi}_0,~ \theta(0,\cdot)=\theta_0>0,&
			\end{cases}
		\end{equation}
		where $\boldsymbol\chi_0:=\hp \bb{u}_0$ and  $\tilde{\boldsymbol\chi}_0:=\hp \bb{v}_0$. In addition, $\boldsymbol\nu$ and $\boldsymbol\chi$ attain the same regularity as $\bb u$. Moreover, we have the following decomposition of the energy
		\begin{align*}
			\ul\io \boldsymbol\nu_t^2dx+\ul\io |\nabla \boldsymbol\nu|^2dx=\ul\io \tilde{\boldsymbol\nu}_0^2dx+\ul\io |\nabla \boldsymbol\nu_0|^2dx
		\end{align*}
		and
		\begin{align*}
			\ul\io \boldsymbol\chi_t^2dx+\ul\io |\nabla \boldsymbol\chi|^2dx+\io\theta\, dx=\ul\io \tilde{\boldsymbol\chi}_0^2dx+\ul\io |\nabla \boldsymbol\chi_0|^2dx+\io\theta_0\, dx.
		\end{align*}
	\end{tw}
	\begin{proof}
		Let us deal with the wave equation from \eqref{new:system}. We multiply it by $\boldsymbol\nu_{tt}-\Delta\boldsymbol\nu$ and integrate over $\tn$. We obtain
		\begin{align*}
			\io |\boldsymbol\nu_{tt}-\Delta\boldsymbol\nu|^2dx+\io(\boldsymbol\chi_{tt}-\Delta\boldsymbol\chi)\cdot (\boldsymbol\nu_{tt}-\Delta\boldsymbol\nu)dx=-\mu\io\nabla\theta\cdot(\boldsymbol\nu_{tt}-\Delta\boldsymbol\nu)dx.
		\end{align*}
		We integrate the right-hand side by parts and see that it disappears due to the definition of $\boldsymbol\nu$. Let us consider the second integral on the left-hand side. We know that $\boldsymbol\chi=\hp \bb u$.
		Therefore, 
		\begin{align*}
			\io(\boldsymbol\chi_{tt}-\Delta\boldsymbol\chi)\cdot (\boldsymbol\nu_{tt}-\Delta\boldsymbol\nu)dx
			= \io \hp (\bb u_{tt} - \Delta \bb u) H(\bb u_{tt} - \Delta \bb u)dx=0.
		\end{align*}
		Thus, we have that
		\begin{align*}
			\io |\boldsymbol\nu_{tt}-\Delta\boldsymbol\nu|^2dx=0,
		\end{align*}
		which implies that $\nu$ satisfies \eqref{eqdecnondiv}.
		
		We know that $\div \bb u=\div\boldsymbol\chi+\div\boldsymbol\nu=\div\boldsymbol\chi$. It shows that the lower equation from \eqref{eqdecnoncurl} is satisfied. The upper is attained when we subtract \eqref{eqdecnondiv} from \eqref{new:system}. Then the energy decomposition follows.
	\end{proof}

	\subsection{Displacement and temperature at infinity}
	
	From the previous section, we know that the displacement decomposes into two different vector fields
	\begin{align*}
		\bb u=\boldsymbol\nu+\boldsymbol\chi.
	\end{align*}
	As time progresses, \(\boldsymbol\nu\) and \(\boldsymbol\chi\) exhibit fundamentally different behaviors. Specifically, the divergence-free component \(\boldsymbol\nu\) remains oscillatory unless its initial value is identically zero. To formalize this observation, we state the following proposition.
	
	\begin{prop}
		Let $(\bb u, \theta)$ be a global solution of \eqref{new:system}. Then,
		\begin{align}\label{connondiv}
			\boldsymbol\nu:=H\bb u\to 0\textrm{ in } \ld, \textrm{ as } t\to\infty
		\end{align}
		if and only if $\boldsymbol\nu_0:=H\bb{u}_0=0$ and $\tilde{\boldsymbol\nu}_0:=H\bb{v}_0=0$.
	\end{prop}
	\begin{proof}
		Due to Theorem \ref{twdecompsol}, the vector field $\boldsymbol\nu$ is a solution of \eqref{eqdecnondiv}. This is the wave equation. If $\boldsymbol\nu_0=0$ and $\tilde{\boldsymbol{\nu}}_0=0$, then $\boldsymbol\nu=0$.
		
		Now, let us assume that \eqref{connondiv} is valid. Let us denote that $\{\boldsymbol\psi_i\}$ is basis of $\ld$ consisting of eigenfunctions of $-\Delta$ with values in $\rn$. Moreover, we assume that $\|\boldsymbol\psi_i\|_{\ld}=1$ for all $i\in\n$. Let us also denote the sequence of eigenvalues of $-\Delta$ on $\tn$ as $\{\lambda_i\}$. Then,  Hence, we can write $\boldsymbol\nu(t, x) =\sum_{i=1}^{\infty} d_i(t) \boldsymbol\psi_i(x)$, where $d_i$ is a solution of the problem
		\begin{align*}
			\begin{cases}
				d_i''+\lambda_id_i=0&\textrm{on } [0,\infty),\\
				d_i(0)=(\boldsymbol\nu_0,\boldsymbol\psi_i)_{\ld},\ d_i'(0)=(\tilde{\boldsymbol\nu}_0,\boldsymbol\psi_i)_{\ld}.
			\end{cases}
		\end{align*}
		Condition \eqref{connondiv} implies that $d_i(t)=(\boldsymbol\nu(t),\boldsymbol\psi_i)_{\ld}\to 0$, as $t\to\infty$ for each $i$. Because $\lambda_i\geq 0$, this could happen only if $d_i=0$. It means that $\boldsymbol\nu=0$ and we derive that $\boldsymbol\nu_0=0$ and $\tilde{\boldsymbol\nu}_0=0$.
	\end{proof}
	We notice that even in the nonhomogeneous case of the wave equation, oscillations prevail almost always, see \cite{Bies} for details.
	
	Let us now focus on the asymptotics of the curl-free part of $\bb u$ (i.e., $\boldsymbol\chi$) and the temperature. Our analysis follows morally the one in \cite{BC2}, however technical details need to be adjusted. As in \cite{BC2}, we start with the stability lemma in appropriate functional spaces.

	\begin{prop}\label{twcontdep}
		Let $\bb{u}_{0,1},\bb{u}_{0,2}\in \htwo, \bb{v}_{0,1},\bb{v}_{0,2}\in\hom$ and $\theta_{0,1},\theta_{0,2}\in\hom\cap L^{\infty}(\tn)$. Moreover, we  require that there exists $\tilde\theta_1>0$ and $\tilde\theta_2>0$ such that $\tilde\theta_1\leq\theta_{0,1}$ and $\tilde\theta_2\leq\theta_{0,2}$.
		Let us also assume that $(\bb{u}_1,\theta_1)$ and $(\bb{u}_2,\theta_2)$ are global solutions of \eqref{new:system} (as obtained in Theorem~\ref{main1}) starting from $\bb{u}_{0,1},\bb{v}_{0,1},\theta_{0,1}$ and $\bb{u}_{0,2},\bb{v}_{0,2},\theta_{0,2}$, respectively.  Then, for any $T > 0$, there exists $C=C(T,d,|\mu|, \|\nabla \bb{u}_{1,t}\|_{L^{\infty}(0,T;\ld)},\|\theta_2\|_{L^{\infty}((0,T)\times\tn)},\|\nabla \theta_2\|_{L^{\infty}(0,T;\ld)}, \|\nabla^2\theta_2\|_{L^2((0,T) \times \tn)})>0$ such that
		\begin{align*}
			&\|\bb{u}_1(t,\cdot)-\bb{u}_2(t,\cdot)\|_{\hom}+\|\bb{u}_{1,t}(t,\cdot)-\bb{u}_{2,t}(t,\cdot)\|_{\ld}+\|\theta_1(t,\cdot)-\theta_2(t,\cdot)\|_{\ld}\\
			&\quad\leq C\left(\|\bb{u}_{0,1}-\bb{u}_{0,2}\|_{\hom}+\|\bb{v}_{0,1}-\bb{v}_{0,2}\|_{\ld}+\|\theta_{0,1}-\theta_{0,2}\|_{\ld}\right)
		\end{align*}
		for all $t \in [0, T)$.
	\end{prop}
	\begin{proof}
		Let us denote $\bb{u}:=\bb{u}_1-\bb{u}_2$ and $\theta:=\theta_1-\theta_2$. We subtract the equations for $(\bb{u}_2,\theta_2)$ from the equations for $(\bb{u}_1,\theta_1)$ and
		obtain
		\begin{align}\label{uklpom}
			\begin{cases}
				\bb{u}_{tt}-\Delta \bb{u}=-\mu\nabla\theta,\\
				\theta_t-\Delta\theta=-\mu\theta \div \bb{u}_{1,t}-\mu \theta_2\div \bb{u}_{t}.
			\end{cases}
		\end{align}
		We multiply by $\bb{u}_t$ the first of the above equations and arrive at
		\begin{align}\label{glrow1}
			\ul\ddt\left(\io \bb{u}_t^2dx+\io |\nabla u|^2dx \right)=-\mu\io \bb{u}_t\cdot\nabla\theta dx\leq C\|\bb{u}_t\|_{\ld}^2+\epsilon\|\nabla \theta\|^2_{\ld}.
		\end{align}
		Whereas, multiplying the second equation in \eqref{uklpom} by $\theta$, we obtain
		\begin{align}\label{glrow2}
			\nonumber\ul\ddt&\io\theta^2 dx+\io|\nabla\theta|^2 dx=-\mu\io\theta^2\div \bb{u}_{1,t}dx-\mu\io\theta\theta_2\div \bb{u}_{t}dx\\
			&=-\mu\io\theta^2\div \bb{u}_{1,t}dx+\mu\io\theta\nabla\theta_2\cdot \bb{u}_{t}dx+\mu\io\theta_2\nabla\theta\cdot \bb{u}_{t}dx=I_1+I_2+I_3.
		\end{align}
		
		We bound the integral $I_1$ using the Gagliardo-Nirenberg inequality and then the Young inequality.
		We get that for each $\epsilon$, there exists $C = C(|\mu|, \nabla \bb{u}_{1,T}\|_{L^{\infty}(0,T;\ld)}) > 0$ such that
		\begin{align}\label{inconde1}
			I_1\leq |\mu|\|\div \bb{u}_{1,t}\|_{L^{\infty}(0,T;\ld)}\|\theta\|_{L^4(\tn)}^2\leq \epsilon\|\nabla\theta\|^2_{\ld}+C\|\theta\|_{\ld}^2.
		\end{align}
		The integral $I_{3}$ we estimate by the Schwarz inequality and the Young inequality. We also utilize Lemma \ref{upbofortemp} here
		and for any $\epsilon > 0$ find $C = C(|\mu|, \|\theta_2\|_{L^{\infty}((0,T)\times\tn)})$ such that
		\begin{align}\label{inconde2}
			I_{3}\leq |\mu|\|\theta_2\|_{L^{\infty}((0,T)\times\tn)}\|\nabla\theta\|_{\ld}\|\bb{u}_t\|_{\ld}\leq \epsilon\|\nabla\theta\|_{\ld}^2+C\|\bb{u}_t\|_{\ld}^2.
		\end{align}
		Let us proceed with $I_{2}$. We use the Gagliardo-Nirenberg inequality and the Young inequality.
		\begin{align}\label{inconde3}
			I_{2}&\leq |\mu| \|\theta\|_{L^4(\tn)}\|\nabla\theta_2\|_{L^4(\tn)}\|\bb{u}_t\|_{\ld} \notag \\
			&\leq C(\|\nabla\theta\|_{\ld}+\|\theta\|_{\ld})(\|\nabla^2\theta_2\|_{\ld}+\|\nabla\theta_2\|_{\ld})\|\bb{u}_t\|_{\ld} \notag \\
			&\leq \epsilon\|\nabla\theta\|_{\ld}^2+ C(\|\nabla^2\theta_2\|_{\ld}^2\|\bb{u}_t\|_{\ld}^2+\|\bb{u}_t\|_{\ld}^2+\|\theta\|_{\ld}^2).
		\end{align}
		
		We plug \eqref{inconde1}, \eqref{inconde2} and \eqref{inconde3} into \eqref{glrow2}. The obtained inequality we add to \eqref{glrow1}. It leads us to
		\begin{align}\label{inconde_combined}
			\ul\ddt&\left(\|\bb{u}_t\|^2_{\ld}+\|\nabla\bb u\|_{\ld}^2+\|\theta\|^2_{\ld}\right)+\|\nabla\theta\|_{\ld}^2 \notag
			\\ &\leq4\epsilon\|\nabla\theta\|_{\ld}^2+ C(\|\nabla^2\theta_2\|_{\ld}^2\|\bb{u}_t\|_{\ld}^2+\|\bb{u}_t\|_{\ld}^2+\|\theta\|_{\ld}^2).
		\end{align}
		We take $\epsilon:=\frac18$.
		
		Let us denote $y:=\|\bb{u}_t\|^2_{\ld}+\|\nabla \bb u\|_{\ld}^2+\|\theta\|^2_{\ld}$. Then, we rewrite \eqref{inconde_combined} as
		\begin{align}\label{inconde_odi}
			y'\leq C y(1+\|\nabla^2\theta_2\|_{\ld}^2).
		\end{align}
		Multiplying \eqref{inconde_odi} by $\operatorname{exp}\left(-C\int_0^s\|\nabla^2\theta_2\|_{\ld}^2d\lambda-Cs\right)$ for $s\in(0,t]$, we obtain
		\begin{align*}
			\frac{d}{ds}\left(y \operatorname{exp}\left(-C\int_0^s\|\nabla^2\theta_2\|_{\ld}^2d\lambda-Cs\right)\right)\leq 0.
		\end{align*}
		We integrate over $[0,t]$ and see that
		$$
		y(t)\leq y(0)\operatorname{exp}\left(C\int_0^t\|\nabla^2\theta_2\|_{\ld}^2ds+Ct\right).
		$$
		It gives the thesis.
	\end{proof}
	
	We denote
	\begin{align*}
		\M:=\left\{\eta\in \hom \cap L^\infty(\tn) \colon \operatorname*{ess\, inf}_{x\in\tn}\eta(x)>0\right\}.
	\end{align*}
	Let us take $t\geq 0$. We define an operator
	\begin{align*}
		S(t):\htwo\times\hom\times\M\to \htwo\times\hom\times\M
	\end{align*}
	by formula
	\begin{align*}
		S(t)(\bb{u}_0,\bb{v}_0,\theta_0)=(\bb u(t),\bb{u}_t(t),\theta(t)),
	\end{align*}
	where $(\bf u, \theta)$ is a solution of \eqref{new:system} with initial values $\bb{u}_0, \bb{v}_0$ and $\theta_0$. Thanks to the uniqueness of the solutions of the system, we obtain
	\begin{align*}
		S(t+h)=S(h)\circ S(t)
	\end{align*}
	for all $t,h\in [0,\infty)$.
	
	Proposition \ref{twcontdep} implies the following important corollary.
	\begin{cor}\label{corcontdep}
		Let  $(\bb{u}_0,\bb{v}_0,\theta_0),(\bb{u}_n,\bb v_n,\theta_n)\in \htwo\times\hom\times\M$ be such that
		\begin{align*}
			(\bb{u}_n,\bb v_n,\theta_n)\to(\bb{u}_0,\bb{v}_0,\theta_0)\textrm{ in }\hom\times\ld\times\ld.
		\end{align*}
		Moreover, assume that all the quantities occurring in the constant $C$ in Proposition \ref{twcontdep} are controlled along $(\bb{u}_n,v_n,\theta_n)$. The latter means that we are restricted to a subset of a phase space. Then, we have
		\begin{align*}
			S(t)(\bb{u}_n,\bb v_n,\theta_n)\to S(t)(\bb{u}_0,\bb{v}_0,\theta_0)\textrm{ in }\hom\times\ld\times\ld
		\end{align*}
		for each $t\geq 0$.
	\end{cor}
	The above corollary will be applied to the sequence of values of the solution in times $t_k\rightarrow \infty$. In particular, since this is the same solution, the control of quantities appearing in the constant $C$ in Proposition \ref{twcontdep} is ensured.

	We need certain global estimates for $\boldsymbol\nu$ and $\theta$. We will prove them in the below theorem.
	\begin{tw}\label{twestas}
		Suppose that the initial data fulfill \eqref{initval} and \eqref{initcondasymp}.
		Then the global solution $(\bb u, \theta)$ of \eqref{new:system} given by Theorem~\ref{main1} fulfills
		\begin{align*}
			\boldsymbol\chi:=\hp \bb u&\in L^{\infty}(0,\infty;\htwo),\ \boldsymbol\chi_t \in L^{\infty}(0,\infty;\hom),\ \boldsymbol\chi_{tt} \in L^{\infty}(0,\infty;\ld) \textrm{ and }\\ 
			\theta&\in C([0,\infty);\hom)\cap L^{\infty}(0,\infty;\hom).
		\end{align*}
	\end{tw}
	\begin{proof}
		By Lemma~\ref{glob:est}, we get
		\begin{align}\label{inestasa}
			&\nabla \bb{u}_t\in L^{\infty}(0,\infty;\ld), \quad
			\Delta \bb{u}\in L^{\infty}(0,\infty;\ld), \quad
			\nabla\theta^{\sul}\in L^{\infty}(0,\infty;\ld), \\
			&\theta \in L^{2}_{\loc}(0,\infty;\htwo) \quad \textrm{and} \quad
			\theta_t \in L^{2}_{\loc}(0,\infty;\ld) \notag
		\end{align}
		Hence, $\theta \in C([0, \infty); \hom)$. The Sobolev embeddings in dimensions 2 and 3 yield the existence of a time $t_0$ such that $u(t_0, \cdot)$ is bounded (because $\theta \in L^{2}_{\loc}(0,\infty;\htwo)$). 
		Therefore, Lemma \ref{upbofortemp} guarantees that $\theta\in L^{\infty}((t_0,\infty)\times\tn)$. Moreover, we obtain
		\begin{align*}
			\|\nabla\theta\|_{\ld}^2=\io|\nabla\theta|^2 dx\leq \|\theta\|_{L^{\infty}((t_0,\infty)\times\tn)}4\io|\nabla\theta^{\sul}|^2dx
		\end{align*}
		for $t\in (t_0,\infty)$. 
		Consequently, owing to \eqref{inestasa}, we observe that $\nabla\theta \in L^{\infty}(t_0,\infty;\ld)$. By \eqref{balance}, we  have that $\theta\in L^{\infty}(0,\infty;L^1(\tn))$. Consequently, $\theta\in L^{\infty}(t_0,\infty;\hom)$. This together with $\theta \in C([0, \infty); \hom)$ gives that $\theta\in L^{\infty}(t_0,\infty;\hom)$.
		
		Because of \eqref{inestasa}, we know that
		\begin{align}\label{inestas11}
			\nabla \boldsymbol\chi_t\in L^{\infty}(0,\infty;\ld),\
			\Delta \boldsymbol\chi\in L^{\infty}(0,\infty;\ld).
		\end{align}
		Thanks to Theorem~\ref{thdecomp},
		\begin{align}\label{inestas12}
			\io \boldsymbol\chi dx=0.
		\end{align}
		Therefore, the Poincar\'e inequality and Lemma~\ref{inpion} give us that
		\begin{align}\label{inestas2}
			\|\boldsymbol\chi\|_{\ld}\leq C \|\nabla\boldsymbol\chi\|_{\ld}\leq C\|\Delta\boldsymbol\chi\|_{\ld}.
		\end{align}
		From the theory of regularity of solutions for elliptic equations (see \cite[Section 6.3]{evans}), we have the following inequality
		\begin{align*}
			\|\boldsymbol\chi\|_{\htwo}\leq C(\|\Delta\boldsymbol\chi\|_{\ld}+\|\boldsymbol\chi\|_{\ld}).
		\end{align*}
		This with \eqref{inestas2} and \eqref{inestas11} gives us that $\boldsymbol\chi\in L^{\infty}(0,\infty;\htwo)$. Thanks to the Poincar\'e inequality \eqref{inestas12} and \eqref{inestas11}, we also show that $\boldsymbol\chi_t\in L^{\infty}(0,\infty;\hom)$. Due to Theorem \ref{thdecomp}, we know that the equation
		\begin{align*}
			\boldsymbol\chi_{tt}-\Delta\boldsymbol\chi=-\mu\nabla\theta
		\end{align*} 
		is satisfied. We already know that $\Delta\boldsymbol\chi,\nabla\theta\in L^{\infty}(0,\infty,\ld)$. It implies that $\boldsymbol\chi_{tt} \in L^{\infty}(0,\infty;\ld)$.
	\end{proof}
	
	Now, we are in a position to formulate and prove the main theorem of this section and hence also Theorem~\ref{main3}. It follows closely the proof in \cite{BC2}.
	However, $t \mapsto \boldsymbol\chi(t,\cdot)$ and $t \mapsto \boldsymbol\chi_t(t,\cdot)$ are only known to be essentially bounded in $\htwo$ and $\hom$, respectively.
	This makes the compactness argument presented below a bit more delicate.
	\begin{tw}
		Suppose that the initial data fulfill \eqref{initval}, \eqref{initcondasymp} and $\theta_0^{-1} \in L^\infty(\tn)$.
		Then for the global solution $(\bb u, \theta)$ of \eqref{new:system} given by Theorem~\ref{main1}, the following convergences hold:
		\begin{align*}
			\boldsymbol\chi(t,\cdot)&\to 0\textrm{ in }\hom \textrm{ as }t\to\infty,\\
			\boldsymbol\chi_t(t,\cdot)&\to 0\textrm{ in }\ld \textrm{ as }t\to\infty,\\
			\theta(t,\cdot)&\to \theta_{\infty}\textrm{ in }\ld \textrm{ as }t\to\infty,
		\end{align*}
		where $\theta_{\infty}:=\left(\ul\io \boldsymbol{\tilde\chi}_0^2dx+\ul\io |\nabla\boldsymbol\chi_0|^2dx+\io\theta_0 dx\right)$,
		$\boldsymbol\chi=\hp \bb u$, $\boldsymbol\chi_0=\hp \bb{u}_0$ and $\boldsymbol{\tilde\chi}_0=\hp \bb{v}_0$.
	\end{tw}
	\begin{proof}
		Without loss of generality, we can consider $\theta_0\in L^\infty(\tn)$. Indeed, in view of Lemma~\ref{local:est}, we know that $\theta\in L^2_{\loc}(0,\infty;\htwo)$. Hence, the Sobolev embeddings in dimensions 2 and 3 yield the existence of a time $t_0$ such that $\theta(t_0, \cdot)$ is bounded. Let us take a sequence $t_n\in [0,\infty)$ such that $t_n\to\infty$. We consider the sequences $\boldsymbol\chi(t_n,\cdot)$, $\boldsymbol\chi_t(t_n,\cdot)$ and
		$\boldsymbol\chi(t_n,\cdot)$. We know by Theorem~\ref{twestas} that
		\begin{align}\label{cont}
			\begin{split}
				\boldsymbol\chi\in C([0,\infty);&\hom),\ \boldsymbol\chi_t \in C([0,\infty);\ld),\\
				\theta& \in C([0,\infty);\hom),
			\end{split}
		\end{align}
		hence the sequence $\theta(t_n,\cdot)$ is bounded in $\hom$.
		However, we only know that
		\begin{align*}
			\boldsymbol\chi\in L^{\infty}(0,\infty;\htwo),\quad \boldsymbol\chi_t\in L^{\infty}(0,\infty,\hom).
		\end{align*}
		Thus, there exists a measurable set $\T\subset [0,\infty)$ such that $\left|[0,\infty)\setminus\T\right|=0$ and $\boldsymbol\chi|_{\T}$, $\boldsymbol\chi_t|_{\T}$ are bounded
		functions with values in $\htwo$ and $\hom$ respectively.
		However, if we show
		\begin{align*}
			\lim_{\substack{t\to\infty\\t\in\T}}\boldsymbol\chi(t,\cdot)=0\textrm{ in }\hom
			\quad \text{and} \quad
			\lim_{\substack{t\to\infty\\t\in\T}}\boldsymbol\chi_t(t,\cdot)=0\textrm{ in }\ld,
		\end{align*}
		then \eqref{cont} will imply that
		\begin{align*}
			\lim_{t\to\infty}\boldsymbol\chi(t,\cdot)=0\textrm{ in }\hom,
			\quad \text{and} \quad
			\lim_{t\to\infty}\boldsymbol\chi_t(t,\cdot)=0\textrm{ in }\ld.
		\end{align*}
		
		Therefore, we assume that $\boldsymbol\chi(t_n,\cdot)$ and $\boldsymbol\chi_t(t_n,\cdot)$ are bounded in $\htwo$ and $\hom$, respectively. We also know that $\theta(t_n,\cdot)$ is
		bounded in $\hom$. Thus, we have a subsequence of $t_n$ (still denoted as $t_n$) such that
		\begin{align}\label{cztery}
			\boldsymbol\chi(t_n,\cdot)&\rightharpoonup \boldsymbol{\chi}_\infty\textrm{ in }\htwo,\nonumber\\
			\boldsymbol\chi_t(t_n,\cdot)&\rightharpoonup \boldsymbol{\tilde v}_\infty\textrm{ in }\hom,\\
			\theta(t_n,\cdot)&\rightharpoonup\theta_{\infty}\textrm{ in }\hom,\nonumber
		\end{align}
		where $\boldsymbol{\chi}_\infty\in\htwo$, $\boldsymbol{\tilde v}\in\hom$ and $\theta_{\infty}\in\hom$ are certain functions. 
		By compact embeddings, we have that
		\begin{align*}%\label{conv}
			\begin{split}
				\boldsymbol\chi(t_n,\cdot)&\to \boldsymbol{\chi}_\infty\textrm{ in }\hom,\\
				\boldsymbol\chi_t(t_n,\cdot)&\to \boldsymbol{\tilde v}_\infty\textrm{ in }\ld,\\
				\theta(t_n,\cdot)&\to\theta_{\infty}\textrm{ in }\ld.
			\end{split}
		\end{align*}
		We know that $\boldsymbol\chi(t_n,\cdot)\in\hp(\htwo)$ and $\boldsymbol\chi(t_n,\cdot)\in\hp(\hom)$ for all $n$. Because $\hp(\htwo)$ and $\hp(\hom)$ are convex and closed subsets of $\htwo$ and  $\hom$, respectively, we obtain that $\boldsymbol{\chi}_\infty\in\hp(\htwo)$ and $\bb v_\infty\in\hp(\hom)$. 
		Next, we define solutions starting from functions, which are obtained as limits of $\boldsymbol\chi(t_n,\cdot),\boldsymbol\chi_t(t_n,\cdot)$ and $\theta(t_n,\cdot)$. We denote by $\bar{\boldsymbol\chi}$ and
		$\bar\theta$ the weak solution of \eqref{new:system} with the initial values $\boldsymbol{\chi}_\infty$, $\bb v_\infty$ and $\theta_{\infty}$ (they can be defined due to the fact
		that the limiting objects have enough regularity to define weak solutions, see \eqref{cztery}) and Theorem~\ref{main1};
		note that the lower bound for $\theta$ obtained in Theorem~\ref{thlowbotemp} implies that bounds for $\theta_\infty$ carry over to $\log \theta_\infty$.)
		
		Let us denote $\bar{\boldsymbol\nu}:=H\bar{\boldsymbol\chi}$. Then, the function $\bar{\boldsymbol\nu}$ is a solution of the following problem
		\begin{align*}
			\begin{cases}
				\bar{\boldsymbol\nu}_{tt}-\Delta\bar{\boldsymbol\nu}=0,\\
				\bar{\boldsymbol\nu}(0,\cdot)=0,\ \bar{\boldsymbol\nu}_t(0,\cdot)=0.
			\end{cases}
		\end{align*}
		Thus, we have that $\bar{\boldsymbol\nu}=0$ and then $\bar{\boldsymbol\chi}\in\hp(\htwo)$ and $\bar{\boldsymbol\chi}_t\in\hp(\hom)$. Consequently, there exists a function $\phi$ such that $\bar\chi=\nabla\phi$.
		
		Next, we take $h>0$ and consider the sequences $\bb u(t_n+h,\cdot)$, $\bb{u}_t(t_n+h,\cdot)$ and $\theta(t_n+h,\cdot)$. Similarly, as above, there exists
		$\hat{\boldsymbol\chi}\in\htwo$, $\hat{\bb v}\in\hom$ and $\hat\theta\in\hom$ such that
		\begin{align*}
			\boldsymbol\chi(t_n+h,\cdot)&\to \hat{\boldsymbol\chi}\textrm{ in }\hom,\\
			\boldsymbol\chi_t(t_n+h,\cdot)&\to\hat{\bb v}\textrm{ in }\ld,\\
			\theta(t_n+h,\cdot)&\to\hat{\theta}\textrm{ in }\ld.
		\end{align*}
		If necessary, we take a subsequence of $t_n$ in the above convergences. By Corollary \ref{corcontdep}, we obtain
		\begin{align*}
			(\hat{\boldsymbol\chi},\hat{\bb
				v},\hat\theta)&=\lim_{n\to\infty}(\boldsymbol\chi(t_n+h,\cdot),\boldsymbol\chi(t_n+h,\cdot),\theta(t_n+h,\cdot))=\lim_{n\to\infty}S(h)(\boldsymbol\chi(t_n,\cdot),\boldsymbol\chi_t(t_n,\cdot),\theta(t_n,\cdot))\\
			&=S(h)(\tilde{\boldsymbol \chi},\bb v_\infty,\theta_{\infty})=(\bar{\boldsymbol \chi}(h,\cdot),\bar{\boldsymbol \chi}_t(h,\cdot),\bar\theta(h,\cdot)).
		\end{align*}
		
		On the other hand, by \eqref{ogrzprop}, we know that a function
		\begin{align*}
			t\mapsto\io\log\theta(t,x)dx
		\end{align*}
		is non-decreasing. In view of Lemma \ref{upbofortemp}, it is bounded from above. Hence, it is convergent. Thus, sequences $\io\log\theta(t_n,x)dx$ and $\io\log\theta(t_n+h,x)dx$ must be convergent to the same number. Moreover,
		\begin{align}\label{piec}
			\nonumber\io\log\theta_{\infty} dx&=\lim_{n\to\infty}\io\log\theta(t_n,x)dx=\lim_{n\to\infty}\io\log\theta(t_n+h,x)dx\\
			&=\io\log\bar\theta(h,x)dx.
		\end{align}
		Indeed, $\theta(t_n,x)$ is convergent in $L^2$, so also a.e.\ (up to a subsequence). Next, we use Lemma \ref{thlowbotemp} to infer that both
		$\bar{\theta}(h,\cdot)$ and $\theta_{\infty}$ are bounded away from $0$. Lemma \ref{upbofortemp} ensures that $\log\theta$ is also bounded from above.  Hence, we conclude \eqref{piec} via the Lebesgue theorem.
		
		Next, we notice that \eqref{piec} states that
		\begin{align*}
			t\mapsto\io\log\bar\theta(t,x)dx
		\end{align*}
		is constant. The second principle of thermodynamics has just told us that the integral of entropy is constant at the $\omega$-limit set. Further, we
		make use of the knowledge of the dissipation rate of the entropy in \eqref{ogrzprop} to obtain
		\begin{align*}
			0=\ddt\io\log\bar\theta dx=\io\frac{|\nabla \bar\theta|^2}{\bar\theta^2}dx.
		\end{align*}
		Hence,
		\[
		\nabla \left(\log{\bar{\theta}}\right)=0\;\;\mbox{in}\;\;\tn
		\]
		for any $t>0$. So, for each fixed $t>0$, $\bar{\theta}$ is constant in space.
		But then, applying \eqref{ogrzprop} again, we obtain
		\begin{align*}
			0=\ddt\io\log\bar\theta dx=\ddt\log\bar\theta =\frac{\bar\theta_t}{\bar\theta}.
		\end{align*}
		It means that $\bar\theta$ is also constant in time. So, $\bar{\theta}(t,x)$ is constant in space and time.
		
		So far, we identified all the trajectories of \eqref{new:system} starting from functions $\theta_\infty$. We observe now that, actually, the $\omega$-limit set of a solution of \eqref{new:system} contains only one
		temperature $\theta_{\infty}$. To this end, we utilize the second principle of thermodynamics once more. Assume there are two different functions
		$\theta_{\infty}^1$ and $\theta_{\infty}^2$, which are limits of $\theta(t_{n_k},\cdot)$ and $\theta(t_{n_l},\cdot)$, respectively, for different time sequences
		$t_{n_k}, t_{n_l}\rightarrow \infty$. Owing to  \eqref{ogrzprop}, and following an analogous argument as above, we obtain:
		\[
		\io \log \left(\theta_{\infty}^1\right)dx =\io \log \left(\theta_{\infty}^2\right)dx.
		\]
		Moreover, we already know that solutions starting from $\theta_{\infty}^1$ and $\theta_{\infty}^2$ are constant in space and time. Hence
		\[
		\log (\theta_{\infty}^1)=\log (\theta_{\infty}^2)\Longrightarrow \theta_{\infty}^1=\theta_{\infty}^2.
		\]
		Thus, we have just shown that the $\tn$-limit set of a solution to \eqref{new:system} contains a single function, which is constant. Next, we identify all potential limits of the curl-free part of the displacement.
		
		Since $\bar{\theta}$ is constant, we have
		\begin{align}\label{szesc}
			\begin{cases}
				\bar{\boldsymbol\chi}_{tt}-\Delta\bar{\boldsymbol\chi}=0,\\
				0=\mu\bar \theta\div\bar{\boldsymbol\chi}_{t}.
			\end{cases}
		\end{align}
		
		The second equation in \eqref{szesc} show that $\Delta\phi_t=0$ for all $(t,x)\in[0,\infty)\times\tn$. Therefore, the function $\phi_t$ is a harmonic function on $\tn$. Due to the periodicity and the maximum principle, this is only possible if $\phi_t$ is a constant. Hence, $\bar{\boldsymbol\chi}_t=\nabla\phi_t=0$ in $\tn$ and so $\bar{\chi}$ is constant
		in time for each $x\in\tn$. This, together with the first equation in \eqref{szesc}, states that $\Delta\bar{\boldsymbol\chi}=0$, and this implies that $\bar{\boldsymbol\chi}$ is constant. However, we have
		$$\io\bar{\boldsymbol\chi} dx=\io\nabla\phi dx=0,$$
		so $\bar{\boldsymbol\chi}=0$.
		Finally, the conservation of the energy as in Lemma $\ref{uni:est}$ holds for the pair $(\boldsymbol\chi,\theta)$ as well
		\begin{align*}
			\frac12&\io \tilde{\boldsymbol\chi}_0^2dx+\frac12\io |\nabla\boldsymbol\chi_0|^2dx+\io\theta_0 dx\\
			&=\lim_{n\to\infty}\left(\frac12\io \boldsymbol\chi_t^2(t_n,x)dx+\frac12\io |\nabla\boldsymbol\chi(t_n,x)|^2dx+\io\theta(t_n,x) dx\right)\\
			&=\io\theta_{\infty} dx=\theta_{\infty}.
			\qedhere
		\end{align*}
	\end{proof}
	
	\section{The Lam\'e operator}\label{lamsec}
	
	In this section, we shall deal with the following system
	\begin{equation}\label{eqlame}
		\begin{cases}
			\bb{u}_{tt}-(2\zeta+\lambda)\nabla\div\bb{u}+\zeta\curl\curl \bb u =   -\mu\nabla\theta ,& \quad \text{ in } (0,T)\times \tn,\\
			\theta_t - \Delta \theta =- \mu\theta \div\bb{u}_t,& \quad \text{ in }  (0,T)\times\tn,\\
			\bb{u}(0,\cdot)=\bb{u}_0,~ \bb{u}_t(0,\cdot) = \bb{v}_0,~ \theta(0,\cdot)=\theta_0>0.&
		\end{cases}
	\end{equation}
	We consider here the operator $Lu:=-(2\zeta+\lambda)\nabla\div\bb{u}+\zeta\curl\curl \bb u$ instead of $-\Delta\bb u$. The operator is called the Lam\'e operator (see \cite{Racke_book, Slaughter}). The constants $\zeta,\lambda\in\r$, the so-called the Lam\'e moduli, are given. They satisfy
	\begin{align*}
		\zeta>0 \ant 2\zeta+d\lambda>0.
	\end{align*}
	Let us note that in the 2D case the $\curl\curl \bb u$ is understood as $\nabla^{\perp}\curl\bb u$, where $\curl \bb u:=u^2_{x_1}-u^1_{x_2}$ and $\nabla^{\perp} f:=(f_{x_2},-f_{x_1})$.  	
	
	Approximate solutions are constructed exactly the same way as in the case of the Laplacian. Basis in the half-Galerkin procedure consists again of the eigenfunctions of the Laplacian. Next, the energy and functional $\F$ look in the similar way. Below, in the following lemmata, we calculate these quantities.
	\begin{lem}
		Let $(\bb{u}, \theta)$ with $\theta > 0$ be a smooth solution of \eqref{eqlame}. Then, we have the following total energy balance:
		\begin{align}\label{lamen}
			\begin{split}
				\ul\io \bb{u}_t^2(t)dx&+\frac{2\zeta+\lambda}2\io (\div \bb{u})^2(t)dx+\frac\zeta2\io|\curl \bb u|^2(t)dx+\io\theta(t) dx\\
				&=\ul\io \bb{v}_0^2dx+\frac{2\zeta+\lambda}2\io (\div \bb{u_0})^2dx+\frac\zeta2\io|\curl \bb u_0|^2dx+\io\theta_0 dx.
			\end{split}
		\end{align}
	\end{lem}
	\begin{proof}
		Let us multiply the first equation in \eqref{eqlame} by $\bb u_t$ and integrate over $\tn$. Integrating by parts, we obtain
		\begin{align}\label{lameq1}
			\ddt\left(\ul\io \bb{u}_t^2(t)dx+\frac{2\zeta+\lambda}2\io (\div \bb{u})^2(t)dx+\frac\zeta2\io|\curl \bb u|^2(t)dx\right)&=-\mu\io\nabla\theta\cdot\bb u_tdx\nonumber\\
			&=\mu\io\theta\div\bb u_tdx.
		\end{align}
		Next, we integrate the lower equation in \eqref{eqlame} over $\tn$. It yields
		\begin{align*}
			\ddt\io\theta dx=-\mu\io\theta\div\bb u_tdx.
		\end{align*}
		We add the above equality to \eqref{lameq1}. It gives us the claim.
	\end{proof}
	
	The functional $\F$ is defined as follows
	\begin{align}		
		\F(\bb{u},\theta)&=\ul\left(\io |\nabla \bb{u}_t|^2dx+(2\zeta+\lambda)\io|\nabla\div\bb{u}|^2dx+\zeta\io|\curl\curl\bb u|^2dx+\io\frac{|\nabla\theta|^2}{\theta}dx\right).\label{lamfu}
	\end{align}
	It satisfies a similar formula as in Lemma \ref{wonderfulformula}.
	\begin{lem}
		Let $(\bb{u}, \theta)$ with $\theta > 0$ be a smooth solution of \eqref{eqlame}. 
		Then, the equality
		\begin{align*}
			\frac{d}{dt}\F(\bb{u},\theta)=
			-\io\theta|\nabla^2\log\theta|^2dx-\frac\mu2\io\frac{|\nabla\theta|^2}\theta\div \bb{u}_tdx
		\end{align*}
		holds.
	\end{lem}
	\begin{proof}
		First, we note that the calculations for the Fisher information in the proof of Lemma \ref{wonderfulformula} are independent of the wave equation in \eqref{new:system}. Hence, they are also valid for \eqref{eqlame}. Therefore, we have 
		\begin{align}\label{feq1}
			\ul\ddt\io\frac{|\nabla\theta|^2}{\theta}dx=-4\io\theta|\nabla^2\log\theta^{\sul}|^2dx-\frac\mu2\io\frac{|\nabla\theta|^2}\theta\div \bb{u}_tdx+\mu\io\Delta\theta\div \bb{u}_tdx.
		\end{align}
		Next, we multiply the upper equation in \eqref{eqlame} by $-\Delta\bb u_t$ and integrate over $\tn$. It results in 
		\begin{align*}
			&\io\bb u_{tt}\cdot(-\Delta\bb u_t)-(2\zeta+\lambda)\io\nabla\div\bb u\cdot(-\Delta\bb u_t)dx+\zeta\io\curl\curl\bb u\cdot(-\Delta \bb u_t)dx \\
			& \quad =-\mu\io\nabla\theta\cdot(-\Delta \bb u_t)dx.
		\end{align*}
		We utilize here the equality $-\Delta\bb w=-\nabla\div\bb w+\curl\curl\bb w$ and integrate by parts. It yields
		\begin{align*}
			\ul\ddt\left(\io|\nabla\bb u_t|^2dx+(2\zeta+\lambda)\io|\nabla\div\bb u|^2+\zeta\io|\curl\curl\bb u|^2dx\right)&=\mu\io\nabla\theta\cdot\nabla\div\bb u_tdx\\
			&=-\mu\io\Delta\theta\div\bb u_tdx.
		\end{align*}
		We add the obtained equality to \eqref{feq1} and get the thesis.
	\end{proof}
	
	We notice that the reasoning making use of the heat equation only, see for instance Lemma \ref{moser} or the second principle of thermodynamics \eqref{ogrzprop}, is still valid in the Lam\'e operator case. In the end, a slight modification of the proofs, utilizing \eqref{lamen} and \eqref{lamfu} lead to the following theorem.
	\begin{tw}[Existence, uniqueness and asymptotic behavior]
		Let $d \in \{2, 3\}$.
		\begin{enumerate}
			\item \textbf{Global solution}:
			There is $D > 0$ depending only on $\mu,d$ and $\mathbb{T}^d$ such that if the initial data satisfy
			\begin{align}\label{initcondl}
				\bb{u}_0\in \htwo,\quad \bb{v}_0\in\hom,\quad \theta_0\in\hom, \quad \log\theta_0 \in L^1(\mathbb{T}^d), \quad \nabla\sqrt{\theta_0} \in L^2(\mathbb{T}^d)
			\end{align}
			and
			\begin{eqnarray}\label{initcondl1}
				\|\bb v_0\|_{\ld}^2+(2\zeta+\lambda)\|\nabla\div\bb u_0\|_{\ld}^2+\|\curl\curl\bb u_0\|_{\ld}^2+\left\|\frac{\nabla\theta_0}{\sqrt{\theta_0}}\right\|^2_{\ld}\leq D,
			\end{eqnarray}
			then there exists a unique solution $(\bb u, \theta)$ of the following regularity
			\begin{align*}
				\bb{u}&\in L_{\loc}^{\infty}([0,\infty);\htwo)\cap W_{\loc}^{1,\infty}([0,\infty);\hom)\cap W_{\loc}^{2,\infty}([0,\infty);\ld),\nonumber\\
				\theta&\in L_{\loc}^2([0,\infty);\htwo)\cap H_{\loc}^1([0,\infty); \ld), \nonumber
			\end{align*}
			and
			\begin{align*}
				&\nabla\div \bb{u}\in L^{\infty}(0,\infty;\ld),\quad \curl\curl\bb{u}\in L^{\infty}(0,\infty;\ld), \quad\nabla\bb{u}_t\in L^{\infty}(0,\infty;\ld), \\
				&\sqrt{\theta}\in L^{\infty}(0,\infty;H^1(\mathbb{T}^d)), \quad 
				\log\theta \in L^\infty(0,\infty;L^1(\mathbb{T}^d)), \quad \nabla\log\theta\in L^{2}(0,\infty;\ld).
			\end{align*} 
			
			of \eqref{eqlame}. 
			
			\item \textbf{Local solution}:
			If the initial data satisfy \eqref{initcondl} but not necessarily \eqref{initcondl1},
			then there exists $T \in (0, \infty)$ depending on $\mu, d, \zeta, \lambda$ and $\mathbb{T}^d$ and a unique solution $(\bb u, \theta)$ of \eqref{eqlame} on $(0, T)$.
			\item \textbf{Long-time dynamics:}
			Let us assume that $(\bb u, \theta)$ is a solution of \eqref{eqlame}
			starting from initial data $\bb u_0$, $\bb v_0$ and $\theta_0>0$, satisfying \eqref{initcondl} and $\theta_0^{-1}\in L^{\infty}(\tn)$.
			Then, the following
			convergences hold:
			\begin{alignat*}{2}
				\boldsymbol\chi(t,\cdot)&\to0 &&\quad\textrm{in }\hom \textrm{ as }t\to\infty,\\
				\boldsymbol\chi_t(t,\cdot)&\to 0 &&\quad\textrm{in }\ld \textrm{ as }t\to\infty,\\
				\theta(t,\cdot)&\to \theta_{\infty} &&\quad\textrm{in }\ld \textrm{ as }t\to\infty,
			\end{alignat*}
			where $\theta_{\infty}:=\left(\ul\io \tilde{\boldsymbol\chi}_0^2dx+\frac{2\zeta+\lambda}2\io(\div\boldsymbol\chi_0)^2dx+\io\theta_0 dx\right)$ and $\boldsymbol\chi$ is a curl-free part of $\bb u$, $\boldsymbol\chi_0$ is a curl-free part of $\bb u_0$ and  $\tilde{\boldsymbol\chi}_0$ is a curl-free part of $\bb v_0$.
		\end{enumerate}
	\end{tw}
	
	\noindent
	\textbf{Acknowledgment}.
	BM was supported by the Croatian Science Foundation under project number IP-2022-10-2962. ST gratefully acknowledges the financial support of the Ministry of Science, Technological Development and Innovation of the Republic of Serbia (Grants No. 451-03-137/2025-03/200125 $\&$ 451-03-136/2025-03/200125). The other authors have nothing to declare.

\end{document}